%% file: rod_bending_stationary.tex
\documentclass{amsart}

\usepackage[hmarginratio=1:1]{geometry}

\usepackage{amsmath}
\usepackage{amssymb}
\usepackage{amsthm}
\usepackage{esint}

\usepackage{multirow}

\usepackage{datatool}
\usepackage{booktabs}

\usepackage{capt-of}

\usepackage{graphicx}
\usepackage{float}

\usepackage{color}

\usepackage{hyperref}

\usepackage[backend=biber, style=alphabetic, sorting=nyt, url=false, isbn=false, giveninits=true]{biblatex}

\include{macros_ds}

\allowdisplaybreaks

\title[error estimates for the approximation of elastic curves]{Quasi-optimal error estimates for the approximation of stable stationary states
        of the elastic energy of inextensible curves}
        
\author{Sören Bartels}
\address{Abteilung für angewandte Mathematik, Albert-Ludwigs-Universität Freiburg, Hermann-Herder-Str.~10, 79104 Freiburg i.~Br., Germany}
\email{bartels@mathematik.uni-freiburg.de}

\author{Balázs Kovács}
\address{Institute of Mathematics, Paderborn University, Warburgerstr. 100., 33098, Paderborn, Germany}
\email{balazs.kovacs@math.uni-paderborn.de}

\author{Dominik Schneider}
\address{Abteilung für angewandte Mathematik, Albert-Ludwigs-Universität Freiburg, Hermann-Herder-Str.~10, 79104 Freiburg i.~Br., Germany}
\email{dominik.schneider@mathematik.uni-freiburg.de}

\date{\today}


\begin{document}
\begin{abstract}
  We establish local existence and a quasi-optimal error estimate for piecewise cubic minimizers to the bending energy under a discretized inextensibility constraint. In previous research a discretization is used where the inextensibility constraint is only enforced at the nodes of the discretization. We show why this discretization leads to suboptimal convergence rates and we improve on it by also enforcing the constraint in the midpoints of each subinterval. We then use the inverse function theorem to prove existence and an error estimate for stationary states of the bending energy that yields quasi-optimal convergence. We use numerical simulations to verify the theoretical results experimentally.
\end{abstract}

  
\maketitle

\section{Introduction}\label{section:Introduction}

We consider an arc-length parameterized curve $u: I \to \real^d$ with
$I = (a,b) \subset \real$. Since the curvature $\kappa: I \to \real$ of $u$
is given by $\kappa = \vert u'' \vert$, the bending energy $E(u)$ is
given by
\begin{equation*}
  E(u) = \frac{1}{2} \int_I \vert u'' \vert^2 \ dx.
\end{equation*}
Our goal is to find minimizing functions to this bending energy functional under the inextensibility constraint $\vert u' \vert^2 = 1$ and given boundary conditions $u(a) = u_D(a)$, $u' = u_D'$ on $\del I$. We note that other boundary conditions either yield trivial solutions, i.e. straight lines, or lead to certain consistency terms that we aim to avoid in the error analysis. From the first variation of the energy functional we obtain the Euler-Lagrange equation
\begin{equation*}
  0 = \int_I u'' \cdot v'' \ dx
\end{equation*}
for all tangential fields $v$ satisfying homogeneous boundary conditions and the linearized inextensibility constraint $u' \cdot v' = 0$.
This problem can be discretized using piecewise cubic $C^1$ splines for the approximation of the curve, enforcing the inextensibility constraint only at the nodes of the decomposition of $I$. Using the nodal $\P_1$-interpolant $\I_{h,1}$, we can write the discretized constraint as $\I_{h,1}\vert u_h' \vert^2 = 1$. In \cite{Bartels13} this discretization has been used to define a numerical scheme that approximates these discrete solutions using a discretization of the $L^2$ gradient flow of the bending energy $E$. Numerical simulations show that the discrete solutions obtained with this scheme converge linearly towards the continuous solution in $H^2(I)^d$, which is suboptimal since the interpolation error is of quadratic order.
By using the stronger constraint $\I_{h,2} \vert u_h' \vert^2 = 1$ instead, we will be able to prove quasi-optimal quadratic approximation of regular energy stable solutions, which quantitatively complements general convergence theory.
We will also perform numerical experiments to experimentally verify the improvement in convergence rate.

Similar problems have been studied by several authors. \cite{DD09} deals with the approximation of the elastic flow of parameterized curves. \cite{BGN10}  addresses  the case of closed curves and in \cite{DKS02, BGN08, BGN12, DLP14, DLP17, DN24} the numerical approximation of the $L^2$ gradient flow of the bending energy for curves with fixed length has been studied. Further solutions to the above minimization problem are closely related to harmonic maps into the unit sphere, i.e. minimizers of the Dirichlet energy. Noteworthy papers regarding error estimates of harmonic maps are \cite{HXW09}  following the ideas of \cite{BRR82},  where the minimization problem is reformulated as a saddle point problem using Lagrange multipliers and a preconditioned iterative scheme to solve the problem is proposed, as well as \cite{BPW23}, where this saddle point approach and the inverse function theorem are used to derive a quasi optimal error estimate for the finite element approximation of harmonic maps.
Finally we also want to refer to \cite{BDS25-pre}, where we derived a quasi-optimal error estimate for the elastic flow of inextensible curves and a numerical scheme which will also be used for the numerical experiments in  Section  $\ref{section:numerical_approximation}$. 

A notable difference between this paper and \cite{BDS25-pre} is that by reformulating the minimization problem as a saddle point problem and using the inverse function theorem, we not only obtain an error estimate for the minimizing function $u_h$ but for the discrete Lagrange multiplier $\lambda_h$ as well.

This paper has the following structure. In Section \ref{subsection:notation} we introduce some basic notation and assumptions that we use throughout this paper. In Section \ref{section:P1_solutions} we investigate the suboptimal convergence rate of the $\P_1$ discretized constraint. In Section \ref{section:stationary_case} we reformulate the problem as a saddle point problem and use the inverse function theorem stated in Appendix \ref{appendix:inverse Function} to prove local existence of discrete solutions as well as an error estimate for approximating energy stable regular solutions. In Section \ref{section:numerical_approximation} we introduce a discrete scheme to numerically compute those discrete solutions and verify the error estimate experimentally.

\subsection{Notation}
\label{subsection:notation}
The following notation will be used throughout this paper.
Let $I = \bigcup_{i = 1}^M [x_{i-1}, x_i]$ a decomposition
of an interval $I = (a,b) \subset \real$ with $a = x_0 < x_1 < ... 
< x_M = b$. We set $I_i := [x_{i-1}, x_i]$, $h_i = x_i - x_{i-1}$, 
$h = \max_i h_i$ and $\T_h = \{I_i \ \vert \ i = 1,...,M\}$. 
We will further assume that there exists $c > 0$ independent of $h$ such that $h \leq c h_i$ for all $i = 1,...,M$. 
We then define the finite element spaces
\begin{align*}
  \S^{k,l}(\T_h) := \{v_h \in C^l(I) : v_h\vert_J \in \P_k \text{ for all } J \in \T_h \} \subset H^{l+1}(I).
\end{align*}
To deal with boundary values, we
also define the Sobolev spaces with vanishing boundary values
\begin{gather*}
  H^2_D(I) := \{ v \in H^2(I) : v(a) = 0,\ 
  v'\vert_{\del I} = 0\}, \\
  H^1_D(I) := \{ v \in H^1(I) : v(a) = 0 \}, \leer
  H^1_0(I) := \{ v \in H^1(I) : v\vert_{\del I} = 0 \}.
\end{gather*}
We further set $H^{-1}(I) := H^1_0(I)'$ the dual space of $H^1_0(I)$.
Analogously,  for $l \in \{0,1\}$  we define finite element spaces with vanishing boundary values as
\begin{align*}
  \S^{k,l}_D(\T_h) := \S^{k,l}(\T_h) \cap H^{l+1}_D(I), \leer
  \S^{k,0}_0(\T_h) := \S^{k,0}(\T_h) \cap H^1_0(I).
\end{align*}
We write $(\cdot , \cdot)$ and $\Vert \cdot \Vert$ for the  $L^2(I)$-product 
and -norm and $D_h u$ for the elementwise weak derivative of a function $u$.
Also for $i = 1,...,M$ we set $m_i := (x_{i-1} + x_i)/2$ the midpoint
of the interval $I_i$, $\M(\T_h) := \{m_i : i = 1,...,m\}$ and define the sets of associated nodes for $\S^{1,0}(\T_h)$ and $\S^{2,0}(\T_h)$ as
\begin{align*}
  \N_1(\T_h) := \{x_i : i = 0,...,M\}, \leer
  \N_2(\T_h) := \N_1(\T_h) \cup \M(\T_h).
\end{align*}
We then define the cubic $C^1$ interpolant
$\I_{h,3}: C^1(I)^d \to \S^{3,1}(\T_h)^d$ and the continuous quadratic and 
linear interpolants
$\I_{h,2}: C^0(I)^d \to \S^{2,0}(\T_h)^d$,
$\I_{h,1}: C^0(I)^d \to \S^{1,0}(\T_h)^d$ via the identities
\begin{gather*}
  \I_{h,3}v(z) = v(z),\ (\I_{h,3}v)'(z) = v'(z) \kurz \forall z \in \N_1(\T_h) , \\
  \I_{h,2}v(z) = v(z) \kurz \forall z \in \N_2(\T_h), \leer
  \I_{h,1}v(z) = v(z) \kurz \forall z \in \N_1(\T_h).
\end{gather*}
Additionally we define the linear and quadratic interpolants with vanishing boundary conditions $\I_{h,k,0}: C^0(I)^d \to \S^{k,0}_0(\T_h)^d$ via
\begin{gather*}
  \I_{h,k,0}v(z) = 0 \ \forall z \in \del I,\kurz
  \I_{h,k,0}v(z) = v(z)\ \forall z \in \N_k(\T_h) \setminus \del I
\end{gather*}
for $k \in \{1,2\}$.
Further we introduce another interpolant $\J_{h,3}: C^1(I)^d \to \S^{3,1}(\T_h)^d$ defined via
\begin{equation*}
  \J_{h,3}v(x) = v(a) + \int_a^x \I_{h,2}v' \ d\sigma.
\end{equation*}
We note that, according to
Lemma \ref{lemma:interpolation_estimate}, $\J_{h,3}$
satisfies mostly the same interpolation estimate as $\I_{h,3}$, but preserves derivatives at the subinterval midpoints $m_i$.
Based on these interpolants we also define two lumped $L^2$-products on
$C^0(I)^d$ as
\begin{align*}
  (u,v)_{h,1} = \int_I \I_{h,1}(u \cdot v) \ dx, \leer
  (u,v)_{h,2} = \int_I \I_{h,2}(u \cdot v) \ dx ,
\end{align*}
as well as the corresponding lumped norms
$$
\Vert u \Vert_{h,1} = \sqrt{(u,u)_{h,1}}, \leer
\Vert u \Vert_{h,2} = \sqrt{(u,u)_{h,2}}.
$$

\section{Discrete energy minimization}
\label{section:P1_solutions}
In this section we take a closer look at the solutions to the discrete minimization problem to find out why the approximation error converges only with linear rate. We first recall the continuous minimization problem:
\begin{equation}  
  \begin{aligned}
  \text{Find $u \in u_D + H^2_D(I)^d$ that minimizes: }
  \label{equation:bending_energy_minimization}
  E(u) = \frac{1}{2} \int_I \vert u'' \vert^2 \ dx
  \text{\kurz subject to $\vert u' \vert^2 = 1$.}
  \end{aligned}
\end{equation}
It is easy to see that this minimization problem is closely related to the minimization of the Dirichlet energy:
\begin{equation}  
  \label{equation:dirichlet_energy_minimization}
  \text{Find $\widetilde u \in \widetilde u_D + H^1_0(I)^d$ that minimizes: }
  \widetilde E(v) = \frac{1}{2} \int_I \vert v' \vert^2 \ dx
  \text{\kurz subject to $\vert \widetilde u \vert^2 = 1$.}
\end{equation}
A function $u \in u_D + H^2_D(I)^d$ is a solution to the minimization problem (\ref{equation:bending_energy_minimization}) if and only if $u(a) = u_D(a)$ and $\widetilde u = u'$ is a solution to the minimization problem (\ref{equation:dirichlet_energy_minimization}) with boundary conditions $\widetilde u_D = u_D'$.
An analogous result holds for the discrete case. The following proposition shows that piecewise affine functions minimize the Dirichlet energy.

\begin{proposition}\label{prop:harmonic_maps_energy}
  For all $u \in H^1(I)^d$ we have
  \begin{equation*}
    \int_I \vert (\I_{h,1} u)' \vert^2 \ dx
    \leq \int_I \vert u' \vert^2 \ dx ,
  \end{equation*}
  with equality, if and only if $u \in \S^{1,0}(\T_h)^d$. 
  Especially if $u \in \S^{2,0}(\T_h)^d$
  is a minimizer of the Dirichlet energy under the discrete constraint
  $\I_{h,1}(\vert u \vert^2 - 1) = 0$, then $u \in \S^{1,0}(\T_h)^d$.
\end{proposition}
\begin{proof}
  For all $u \in H^1(I)^d$ we have
  \begin{align*}
    \int_I \vert u' \vert^2 \ dx
    &= \int_I \vert (\I_{h,1} u)' \vert^2 \ dx
    + \int_I \vert (u - \I_{h,1} u)' \vert^2 \ dx
    + 2 \int_I (\I_{h,1} u)' \cdot (u - \I_{h,1} u)' \ dx .
  \end{align*}
  Integrating by parts yields that the last term is zero, which proves the estimate with equality if and only if $\Vert (u - \I_{h,1} u)' \Vert_{L^2(I)^d}^2 = 0$, i.e. $u = \I_{h,1}u$.
\end{proof}

As we have explained above, a function $u$ minimizes the bending energy if and only if it satisfies the required boundary conditions and $u'$ minimizes the Dirichlet energy.
Therefore, applying Proposition \ref{prop:harmonic_maps_energy} to $u'$ yields the following
result:

\begin{corollary}\label{corollary:P1_constraint_minimizer}
  For all $u \in H^2(I)^d$, with 
  $$ \J_{h,2} u := u(a) + \int_a^x \I_{h,1} u' \ dx \in \S^{2,0}(\T_h)^d $$ 
  and Proposition \ref{prop:harmonic_maps_energy} we obtain
  \begin{align*}
    \int_I \vert (\J_{h,2} u)'' \vert^2 \ dx
    = \int_I \vert (\I_{h,1} u')' \vert^2 \ dx
    \leq \int_I \vert u'' \vert^2 \ dx.
  \end{align*}
  Therefore minimizers in $H^2(I)^d$ to the bending energy subject to the discrete inextensibility constraint $\I_{h,1}(\vert u' \vert^2 - 1) = 0$ belong to $\S^{2,0}(\T_h)^d$.
\end{corollary}
This result also yields an explanation, why solutions to the discrete bending problem in general only converge linearly and not
quadratically, as for quadratic splines linear convergence in $H^2(I)^d$ is 
already optimal.

\section{Local  error analysis}\label{section:stationary_case}
To obtain optimal convergence we use a technique based on a version of the inverse function theorem  as in  \cite{BPW23} to derive
quasi-optimal error estimates for the approximation of harmonic maps.
For this a suitable reformulation is needed: Find $u \in u_D + H^2_D(I)^d =: \A_1$ that satisfies the inextensibility constraint $\vert u' \vert^2 = 1$ and the Euler-Lagrange equation
\begin{equation}
  \label{equation:bending_problem_pde}
  0 = \int_I u'' \cdot v'' \ dx
\end{equation}
for all $v \in \G(u) := \{v \in H^2_D(I)^d : u' \cdot v' = 0\}$.
Now let $v \in H^2_D(I)^d$  be  arbitrary. We then have that
$w := v - \int_a^x (u' \cdot v') u' \ d\sigma \in \G(u)$ is a suitable test
function for (\ref{equation:bending_problem_pde}) and testing with $w$ yields
\begin{align*}
  0 = \int_I u'' \cdot w'' \ dx
  = \int_I u'' \cdot v'' - \vert u'' \vert^2 u' \cdot v' \ dx
  = \int_I u'' \cdot v'' + \lambda u' \cdot v' \ dx
\end{align*}
with $\lambda = - \vert u'' \vert^2$. Set $X := H^2_D(I)^d \times H^{-1}(I)$
and $\A := (u_D,0) + X$. Then $u$ is stationary for the bending
energy under the inextensibility constraint if and only if $(u, -\vert u'' \vert^2)$ is stationary for the
Lagrange functional $L: \A \to \real$,
\begin{equation*}
  \label{equation:Lagrange_functional}
  L(u,\lambda) := \frac{1}{2} \int_I \vert u'' \vert^2 \ dx + \frac{1}{2} \left\langle \lambda, \vert u' \vert^2 - 1 \right\rangle.
\end{equation*}
We now define $\widetilde F: X \to X'$ as
\begin{equation*}
  \label{equation:definition_F}
  \begin{aligned}
    \widetilde F(\widetilde{u},\lambda)[v,\mu]
    &:= \int_I (\widetilde{u} + u_D)'' \cdot v'' \ dx
    + \left\langle \lambda, (\widetilde{u}' + u_D') \cdot v' \right\rangle
    + \frac{1}{2} \left\langle \mu, \vert \widetilde{u}' + u_D' \vert^2 - 1 \right\rangle \\
    &= \delta L(u,\lambda)[v,\mu]
  \end{aligned}
\end{equation*}
with $u = \widetilde{u} + u_D \in \A_1$. Thus we get that a function
$u \in \A_1$ is a stationary point of $E$ if and only if
$F(u, - \vert u'' \vert^2) = 0$, where $F(u,\lambda) := \widetilde F(u - u_D,
\lambda)$. \\
A similar approach can be chosen to deal
with the discrete problem. We set $u_{D,h} := \I_{h,3} u_D$ and search for $u_h \in \A_{h,1} := u_{D,h} + \S^{3,1}_D(\T_h)^d$ satisfying the discrete inextensibility constraint $\I_{h,2}(\vert u_h' \vert^2 - 1) = 0$ and the Euler--Lagrange equation
\begin{equation*}\label{equation:discrete_P1_problem}
  \begin{aligned}
    \int_I u_h'' \cdot  v_h''  \ dx = 0
  \end{aligned}
\end{equation*}
for all $v_h \in \G_{h}(u_h) = \{v_h \in \S^{3,1}_D(\T_h) : \I_{h,2}(u_h' \cdot v_h') = 0\}$.\\
We now set
$X_h := (\S^{3,1}_D(\T_h), \Vert \cdot \Vert_{H^2(I)^d}) \times
(\S^{2,0}_0(\T_h), \Vert \cdot \Vert_{H^{-1}(I)}) \subset X$ and
$\A_h := (u_{D,h},0) + X_h$. 
We define the discrete Lagrange functional $L: \A_h \to \real$,
\begin{align*}
  L_h(u_h,\lambda_h)
  := \frac{1}{2} \int_I \vert u_h'' \vert^2 
  + \I_{h,2}(\lambda_h(\vert u_h' \vert^2 - 1)) \ dx ,
\end{align*}
and $\widetilde F_h: X_h \to X_h'$,
\begin{align*}
  \widetilde F_h(u_h - u_{D,h}, \lambda_h)[v_h,\mu_h]
  &:= \int_I u_h'' \cdot v_h'' + \I_{h,2} (\lambda_h u_h' \cdot v_h')
  + \frac{1}{2} \I_{h,2}(\mu_h (\vert u_h' \vert^2 - 1)) \ dx \\
  &= \delta L_h (u_h,\lambda_h)[v_h,\mu_h].
\end{align*}
Analogously to the continuous case, we have that $u_h \in \A_{h,1} :=
u_{D,h} + \S^{3,1}_D(\T_h)^d$ is stationary for
the bending energy $E$ under the discrete constraint $\I_{h,2}(\vert u_h'
\vert^2 - 1) = 0$ if and only if for some $\lambda_h \in  \S^{2,0}_0(\T_h)$  we have 
$F_h(u_h,\lambda_h) := \widetilde F_h(u_h - u_{D,h},\lambda_h) = 0$.
Hereby $\lambda_h$ is the unknown
discrete Lagrange multiplier corresponding to the discrete inextensibility
constraint. We want to obtain the existence of a solution $(u_h,\lambda_h)$
as well as an error estimate by applying the following  quantitative  version of the
inverse function theorem to $\widetilde F_h$ and the interpolants of the 
continuous solution $(u,\lambda)$.

\begin{lemma}[Inverse function theorem]\label{lemma:inverseFunction}
  Let $F: X \to X'$ for a Banach space X and assume there exist $\tilde{x} \in 
  X$, $c_L, c_{inv}, \delta, \varepsilon > 0$ such that
  \begin{enumerate}
    \item $\Vert F(\tilde{x}) \Vert_{X'} \leq \delta$,
    \item $F$ is Fréchet differentiable in $B_\varepsilon(\tilde{x})$,
    \item \label{equation:F Diffeomorphism}
    $\Vert DF(\tilde{x})^{-1} \Vert_{L(X',X)} \leq c_{inv}$,
    \item For all $x_1, x_2 \in B_\varepsilon(\tilde x): 
    \Vert DF(x_1) - DF(x_2) \Vert_{L(X,X')}
    \leq c_L \Vert x_1 - x_2 \Vert_X$,
    \item $c_L c_{inv} \varepsilon \leq \frac{1}{2}, \ 
    2 c_{inv} \delta \leq \varepsilon $.
  \end{enumerate}
  Then there exists a unique $x \in B_\varepsilon(\tilde{x})$ with 
  $F(x) = 0$.
\end{lemma}

\begin{proof}
  See Appendix \ref{appendix:inverse Function}.
\end{proof}

Assume $u \in H^4(I)$ is stationary for the bending energy under the
constraint $\vert u'\vert^2 = 1$ and let $\lambda \in H^2(I)$ be
the corresponding Lagrange multiplier, i.e. $(u,\lambda)$ is
stationary for the continuous saddle-point functional $L$.
Set
\begin{align*}
  \widetilde{u}_h := \I_{h,3}(u), \leer \widetilde{\lambda}_h := 
  \I_{h,2,0}(\lambda).
\end{align*}
We will show that for $h$ small enough, $\widetilde F_h$ and $\widetilde x = (\widetilde{u}_h, \widetilde{\lambda}_h)$ satisfy the
conditions of Lemma \ref{lemma:inverseFunction} with $\varepsilon = c h^2$,
thus yielding the existence of a solution to $F_h(u_h,\lambda_h) = 0$ and a quasi-optimal
error estimate $\Vert u_h - u \Vert_{H^2(I)^d} \leq c h^2$.
Since the definition of $F_h$ involves the term $(\cdot, \cdot)_{h,2}$, we 
first need an estimate for this lumped $L^2$-product.

\begin{lemma}[Quadrature control]\label{lem:quadControl}
  Let $\psi_h \in \S^{2,0}(\T_h)$ and $\Phi \in C^0(I)$ with $\Phi\vert_{I_i} 
  \in H^3(I_i)$
  for all $i = 1,...,M$. Then we have
  \begin{equation*}
    \vert (\psi_h,\Phi) - (\psi_h ,\Phi)_{h,2} \vert
    \leq c h^3 (\Vert \psi_h \Vert \Vert D_h^3 \Phi \Vert
    + \Vert \psi_h' \Vert \Vert D_h^2 \Phi \Vert
    + \Vert D_h \psi_h' \Vert \Vert \Phi' \Vert).
  \end{equation*}
\end{lemma}
\begin{proof}
  Since $\psi_h$ is elementwise $\mathcal{P}_2$ and $\Phi$ is elementwise $H^3$,
  $\psi_h \Phi$ is elementwise $H^3$ as well. Since $\psi_h \Phi$ is also
  continuous, $\I_{h,2}(\psi_h \Phi)$ is well defined and  applying the interpolation estimate of Lemma~\ref{lemma:interpolation_estimate} elementwise yields 
  \begin{align*}
    \vert (\psi_h,\Phi) - (\psi_h ,\Phi)_{h,2} \vert
    &\leq \int_I \vert \psi_h \Phi - \I_{h,2}(\psi_h \Phi) \vert \ dx \\
    &\leq c h^3 \int_I \vert D_h^3 (\psi_h \Phi) \vert \ dx \\
    &\leq c h^3 (\Vert \psi_h \Vert \Vert D_h^3 \Phi \Vert
    + \Vert D_h \psi_h \Vert \Vert D_h^2 \Phi \Vert
    + \Vert D_h^2 \psi_h \Vert \Vert D_h \Phi \Vert).
  \end{align*}
  In the last step we have used the product rule, Hölder's inequality as well as the fact that $D_h^k \psi_h = 0$ for $k > 2$.
\end{proof}

Now we can show that $F_h(\widetilde u_h, \widetilde \lambda_h)$ is controlled in terms of the mesh size such that $(1)$ of Lemma~\ref{lemma:inverseFunction} holds with $\delta = c h^2$.

\begin{lemma}[Boundedness]\label{lem:boundedness}
  The functional $F_h(\widetilde{u}_h, \widetilde{\lambda}_h)$ satisfies
  \begin{equation*}
    \Vert F_h(\widetilde{u}_h, \widetilde{\lambda}_h) \Vert_{X_h'} \leq c h^2.
  \end{equation*}
\end{lemma}
\begin{proof}
  Let $(v_h,\mu_h) \in X_{h}$. Since $(u,\lambda)$ satisfies $F(u,\lambda) = 0$ we
  get
  \begin{align*}
    \vert F_h(\widetilde{u}_h, \widetilde{\lambda}_h)[v_h, \mu_h] \vert
    &= \vert F_h(\widetilde{u}_h, \widetilde{\lambda}_h)[v_h, \mu_h] 
    - F(u,\lambda)[v_h, \mu_h] \vert \\
    &\leq \vert (\widetilde{u}_h'', v_h'') - (u'',v_h'') \vert
    + \vert (\widetilde{\lambda}_h, \widetilde{u}_h' \cdot v_h')_{h,2}
    - (\lambda, u' \cdot v_h') \vert
    + \vert (\mu_h, \vert \widetilde{u}_h' \vert^2 - 1)_{h,2} \vert \\
    &=: \Rom{1} + \Rom{2} + \Rom{3}.
  \end{align*}
  For $\Rom{1}$ an interpolation estimate shows
  \begin{equation}\label{eq:I}
    \Rom{1} = \vert (\I_{h,3}(u)'' - u'', v_h'') \vert 
    \leq \Vert \I_{h,3}(u)'' - u'' \Vert \Vert v_h'' \Vert
    \leq c h^2 \Vert u \Vert_{H^4(I)^d} \Vert v_h \Vert_{H^2(I)^d}.
  \end{equation}
  For $\Rom{2}$ we first note that $v_h \in \S^{3,1}_D(\T_h)^d$, therefore we have $\widetilde{u}_h' \cdot v_h' = 0$ on $\del I$. Since $\I_{h,2,0}v(z) = \I_{h,2}v(z)$ for all nodes $z \in \N_2(\T_h) \setminus \del I$, we obtain $(\widetilde{\lambda}_h, \widetilde{u}_h' \cdot
  v_h')_{h,2} = (\I_{h,2} \lambda, \widetilde{u}_h' \cdot v_h')_{h,2}$. With
  $\widetilde{\lambda} := \I_{h,2} \lambda$ we therefore get
  \begin{align*}
    \Rom{2} 
    &\leq \vert (\widetilde{\lambda}, \widetilde{u}_h' \cdot v_h')_{h,2}
    - (\widetilde{\lambda}, \widetilde{u}_h' \cdot v_h') \vert
    + \vert (\widetilde{\lambda}, \widetilde{u}_h' \cdot v_h')
    - (\widetilde{\lambda}, u' \cdot v_h') \vert
    + \vert (\widetilde{\lambda}, u' \cdot v_h')
    - (\lambda, u' \cdot v') \vert \\
    &=: \Rom{2}_1 + \Rom{2}_2 + \Rom{2}_3.
  \end{align*}
  To estimate $\Rom{2}_1$ we use Lemma \ref{lem:quadControl} and the fact
  that $D_h^4 v_h = 0$ for $v_h \in \S^{3,1}(\T_h)$. We get
  \begin{align*}
    \Rom{2}_1 
    &\leq c h^3 (\Vert \widetilde{\lambda} \Vert 
    \Vert D_h^3(\widetilde{u}_h' \cdot v_h') \Vert
    + \Vert \widetilde{\lambda}' \Vert \Vert D_h^2(\widetilde{u}_h' \cdot v_h') \Vert
    + \Vert D_h \widetilde{\lambda}' \Vert \Vert D_h (\widetilde{u}_h' \cdot v_h')\Vert)\\
    &\leq c h^3 \Vert \widetilde{\lambda} \Vert
    (\Vert D_h \widetilde{u}_h'' \Vert_{L^\infty(I)^d} \Vert v_h'' \Vert
    + \Vert \widetilde{u}_h'' \vert_{L^\infty(I)^d} \Vert D_h v_h'' \Vert) \\
    &\kurz + ch^3 \Vert \widetilde{\lambda}' \Vert
    (\Vert D_h \widetilde{u}_h'' \Vert \Vert v_h' \Vert_{L^\infty(I)^d}
    + \Vert \widetilde{u}_h'' \Vert_{L^\infty(I)^d} \Vert v_h'' \Vert
    + \Vert \widetilde{u}_h' \Vert_{L^\infty(I)^d} \Vert D_h v_h'' \Vert)\\
    &\kurz + ch^3 \Vert D_h \widetilde{\lambda}' \Vert
    (\Vert \widetilde{u}_h'' \Vert \Vert v_h' \Vert_{L^\infty(I)^d}
    + \Vert \widetilde{u}_h' \Vert_{L^\infty(I)^d} \Vert v_h'' \Vert).
  \end{align*}
  
  Combining the interpolation estimate of Lemma~\ref{lemma:interpolation_estimate} with the inverse estimate of Lemma~\ref{lemma:inverse_estimates}  yields $\Vert \widetilde{\lambda} \Vert_{H^2(I)} \leq c \Vert \lambda \Vert_{H^2(I)}$
  and $\Vert D_h^k \widetilde{u}_h \Vert \leq c \Vert u \Vert_{H^4(I)^d}$ for all $k \leq 4$.
  A Sobolev embedding theorem implies $\Vert D_h^k \widetilde{u}_h \Vert_{L^\infty(I_i)^d}
  \leq c \Vert D_h^{k} \widetilde{u}_h \Vert_{H^1(I_i)^d}$.  From the inverse estimate of Lemma \ref{lemma:inverse_estimates} we have  $\Vert D_h v_h'' \Vert  \leq c h^{-1} \Vert v_h \Vert_{H^2(I)^d}$.
   Combining these estimates yields 
  \begin{equation*}
    \Rom{2}_1 \leq c h^2 \Vert \lambda \Vert_{H^2(I)} \Vert u \Vert_{H^4(I)^d}
    \Vert v_h \Vert_{H^2(I)^d}.
  \end{equation*}
  Next we simply use Hölder's inequality and an interpolation estimate to obtain
  a bound on $\Rom{2}_2$:
  \begin{equation*}
    \Rom{2}_2 = \vert (\widetilde{\lambda}, (\widetilde{u}_h' - u') \cdot v_h') \vert
    \leq \Vert \widetilde{\lambda} \Vert \Vert \widetilde{u}_h' - u' \Vert 
    \Vert v_h' \Vert_{L^\infty(I)^d}
    \leq c h^3 \Vert \lambda \Vert_{H^2(I)} \Vert u \Vert_{H^4(I)^d} 
    \Vert v_h \Vert_{H^2(I)^d}.
  \end{equation*}
  Finally using an interpolation estimate on $\Rom{2}_3$ yields
  \begin{equation*}
    \Rom{2}_3 = \vert (\widetilde{\lambda} - \lambda, u' \cdot v_h' ) \vert
    \leq \Vert \widetilde{\lambda} - \lambda \Vert \Vert u' \cdot v_h' \Vert
    \leq c h^2 \Vert \lambda \Vert_{H^2(I)} \Vert u \Vert_{H^4(I)^d}
    \Vert v_h \Vert_{H^2(I)^d}.
  \end{equation*}
  Combining these three estimates yields
  \begin{equation}
    \label{eq:II}
    \Rom{2} \leq c h^2 \Vert \lambda \Vert_{H^2(I)}
    \Vert u \Vert_{H^4(I)^d} \Vert v_h \Vert_{H^2(I)^d}.
  \end{equation}
  It remains  to estimate $\Rom{3}$. We therefore write
  \begin{align*}
    \Rom{3}
    &= \left\vert \int_I \I_{h,2}(\mu_h(\vert \widetilde{u}_h' \vert^2 
    - \vert u' \vert^2 )) \ dx \right\vert \\
    &\leq \left\vert \int_I (\mu_h(\vert \widetilde{u}_h' \vert^2 
    - \vert u' \vert^2 )) \ dx \right\vert
    + \int_I \vert \I_{h,2}(\mu_h(\vert \widetilde{u}_h' \vert^2 
    - \vert u' \vert^2 ))  - (\mu_h(\vert \widetilde{u}_h' \vert^2 
    - \vert u' \vert^2 )) \vert \ dx \\
    &\leq \left\vert \int_I \mu_h( \widetilde{u}_h' - u') \cdot (\widetilde{u}_h' + u')
    \ dx \right\vert
    + c h^3 \int_I \vert D_h^3(\mu_h( \widetilde{u}_h' - u') 
    \cdot (\widetilde{u}_h' + u')) \vert \ dx \\
    &=: \Rom{3}_1 + c h^3 \Rom{3}_2.
  \end{align*}
  $\Rom{3}_1$ can easily be  controlled  using an interpolation estimate:
  \begin{equation*}
    \Rom{3}_1 
    \leq \Vert \mu_h \Vert_{H^1(I)'}
    \Vert (\widetilde{u}_h' - u') \cdot (\widetilde{u}_h' + u') \Vert_{H^1(I)}
    \leq c h^2 \Vert u \Vert_{H^4(I)^d}^2 \Vert \mu_h \Vert_{H^1(I)'}.
  \end{equation*}
  To bound $\Rom{3}_2$  we use  an inverse  estimate to  obtain
  $\Vert D_h^k \mu_h \Vert \leq c h^{-(k+1)} \Vert \mu_h \Vert_{H^1(I)'}$.
  With this estimate we get
  \begin{align*}
    \Rom{3}_2
    &\leq \Vert \mu_h \Vert 
    \Vert D_h^3((\widetilde{u}_h' - u') \cdot (\widetilde{u}_h' + u')) \Vert
    + c \Vert \mu_h' \Vert 
    \Vert D_h^2((\widetilde{u}_h' - u') \cdot (\widetilde{u}_h' + u')) \Vert \\
    &\kurz + c \Vert D_h^2 \mu_h \Vert
    \Vert D_h((\widetilde{u}_h' - u') \cdot (\widetilde{u}_h' + u')) \Vert \\
    &\leq c \Vert \mu_h \Vert \Vert u \Vert_{H^4(I)^d}^2
    + ch \Vert \mu_h' \Vert \Vert u \Vert_{H^4(I)^d}^2
    + ch^2 \Vert D_h^2 \mu_h \Vert \Vert u \Vert_{H^4(I)^d}^2 \\
    &\leq c h^{-1} \Vert \mu_h \Vert_{H^1(I)'} \Vert u \Vert_{H^4(I)^d}.
  \end{align*}
   In combination,  these two estimates yield
  \begin{equation}\label{eq:III}
    \Rom{3} = \Rom{3}_1 + c h^3 \Rom{3}_2 \leq c h^2 \Vert u \Vert_{H^4(I)^d}^2 \Vert \mu_h \Vert_{H^1(I)'}.
  \end{equation}
  Finally (\ref{eq:I}), (\ref{eq:II}) and (\ref{eq:III})  lead to 
  \begin{align*}
    \vert F_h(\widetilde{u}_h, \widetilde{\lambda}_h)[v_h,\mu_h] \vert
    &\leq c h^2 \Vert u \Vert_{H^4(I)^d} (\Vert \lambda \Vert_{H^2(I)} \Vert
    + \Vert u \Vert_{H^4(I)^d})
    (\Vert \mu_h \Vert_{H^1(I)'} + \Vert v_h \Vert_{H^2(I)^d}) \\
    &\leq c_{\lambda, u} h^2 \Vert (\mu_h, v_h) \Vert_{X_h} ,
  \end{align*}
  which finishes the proof.
\end{proof}

\begin{remark}
  In case we replace the discrete inextensibility constraint $\I_{h,2} \vert u_h' \vert^2 = 1$ with the weaker constraint $\I_{h,1} \vert u_h' \vert^2 = 1$ we only have the reduced estimate
  \begin{align*}
    \Vert F_h(\widetilde u_h, \widetilde \lambda_h) \Vert_{X_h'}
    \leq c h.
  \end{align*}
  This is because from an interpolation estimate and an inverse estimate we only get
  \begin{align*}
    \Vert (\widetilde \lambda, \widetilde u_h' \cdot v_h')_{h,1} - (\widetilde \lambda, \widetilde u_h' \cdot v_h') \vert
    &\leq c h^2 \Vert D_h^2(\widetilde \lambda \widetilde u_h' \cdot v_h') \Vert
    \leq c h^2 \Vert \lambda \Vert_{H^2(I)} \Vert u \Vert_{H^3(I)^d} \Vert v_h \Vert_{H^3(I)^d} \\
    &\leq c h \Vert v_h \Vert_{H^2(I)^d}.
  \end{align*}
  So in case all the other assumptions from Lemma \ref{lemma:inverseFunction} hold, this would imply linear convergence of the approximation error.
\end{remark}


The differentiability of $F_h$ is clearly not a problem, so  next we  focus
on the invertibility of $DF_h(\widetilde{u}_h, \widetilde{\lambda}_h)$. 
For this we use Brezzi's theorem,
\begin{lemma}[Brezzi]\label{lem:Brezzi}
  Let $V,Q$ be Hilbert spaces, $X := V \times Q$. Further assume
  $a: V \times V \to \real$ and $b: X \to \real$ are bounded and bilinear.
  We define $B: V \to Q',\ B(v)[q] := b(v,q)$ and
  $L: X \to X', (v,q) \mapsto (a(v,\cdot) + b(\cdot, q), b(v,\cdot))$
  Then $L$ is an isomorphism if and only if
  \begin{enumerate}
    \item $\exists \alpha > 0: \forall v \in \ker B: a(v,v) \geq \alpha \Vert v \Vert_V^2$,
    \item $\exists \beta > 0: \forall q \in Q: \sup_{v \in V \setminus \{0\} }
    \frac{b(v,q)}{\Vert v \Vert_V} \geq \beta \Vert q \Vert_Q$.
  \end{enumerate}
\end{lemma}
\begin{proof}
  The statement follows from the Lax--Milgram Lemma and \cite[Theorem 1.1]{Brezzi1974}.
\end{proof}
To apply this result we first rewrite $DF: \A \to L(X,X')$ as
\begin{align*}
  DF(u,\lambda)[(v,\mu),(w,\eta)] 
  = a_\lambda(v,w) + b_u(w,\mu) + b_u(v,\eta)
\end{align*}
with
\begin{align*}
  a_\lambda(v,w) := (v'', w'') + (\lambda, v' \cdot w'), \leer
  b_u(v,\mu) := (\mu, u' \cdot v').
\end{align*}
Analogously we rewrite $DF_h: \A_h \to L(X_{h},X_{h}')$ as
\begin{align*}
  DF_h(u_h,\lambda_h)[(v_h,\mu_h),(w_h,\eta_h)]
  &= a_{\lambda_h}(v_h,w_h) + b_{u_h}(w_h, \mu_h) + b_{u_h}(v_h, \eta_h)
\end{align*}
with 
\begin{align*}
  a_{\lambda_h}(v_h, w_h) := (v_h'', w_h'') + (\lambda_h, v_h' \cdot w_h')_{h,2}, \leer
  b_{u_h}(v_h,\eta_h) = (\eta_h, u_h' \cdot v_h')_{h,2}.
\end{align*}
We also define $B_{\widetilde{u}_h}(v_h)[\mu_h] := 
b_{\widetilde{u}_h}(v_h,\mu_h)$ and $B_u(v)[\mu] = b_u(v,\mu)$.
So to show that $DF_h(\widetilde{u}_h, \widetilde{\lambda}_h)$ is an 
isomorphism, according to Lemma \ref{lem:Brezzi} we have to
show that $a_{\widetilde{\lambda}_h}$ is coercive on $\ker B_{\widetilde{u}_h}$ 
and that $b_{\widetilde{u}_h}$ satisfies an inf-sup-condition.
We start showing coercivity.

\begin{lemma}[Coercivity]\label{lem:coercivity}
  Assume there exists $\alpha > 0$ such that $a_\lambda$ satisfies the coercivity
  condition $a_\lambda(v,v) \geq \alpha \Vert v \Vert_{H^2(I)^d}^2$ for all
  $v \in \ker B_u$. Then for $h$ small enough we have
  \begin{equation*}
    a_{\widetilde{\lambda}_h}(v_h, v_h) \geq \frac{\alpha}{2} \Vert v_h \Vert_{H^2(I)^d}^2
  \end{equation*}
  for all $v_h \in \ker B_{\widetilde u_h}$.
\end{lemma}
\begin{proof}
  Let $v_h \in \ker B_{\widetilde{u}_h}$. We define $v^h \in \ker B_u$ by
  \begin{equation*}
    v^h(x) := \int_a^x v_h' - (v_h' \cdot u') u' \ dx + v_h(a)
    = v_h - \int_a^x (v_h' \cdot u') u' \ dx.
  \end{equation*}
  Thus we get
  \begin{align*}
    a_{\widetilde{\lambda}_h}(v_h, v_h) 
    &= a_\lambda(v^h, v^h)
    + (a_{\widetilde{\lambda}_h}(v_h, v_h) - a_\lambda(v^h, v^h))\\
    &\geq \alpha \Vert v^h \Vert_{H^2(I)^d}^2
    - \vert a_{\widetilde{\lambda}_h}(v_h, v_h) - a_\lambda(v^h, v^h) \vert.
  \end{align*}
  For the first term we have
  \begin{align*}
    \alpha \Vert v^h \Vert_{H^2(I)^d}^2 
    &= \alpha \Vert v_h \Vert_{H^2(I)^d}^2
    - \alpha ( v_h - v^h , v_h + v^h )_{H^2(I)^d} \\
    &\geq \alpha \Vert v_h \Vert^2
    - \alpha \Vert v_h - v^h \Vert_{H^2(I)^d} \Vert v_h + v^h \Vert_{H^2(I)^d}.
  \end{align*}
  We note that $\Vert v_h + v^h \Vert_{H^2(I)^d} \leq \Vert v_h \Vert_{H^2(I)^d} 
  + \Vert v^h \Vert_{H^2(I)^d} \leq c \Vert v_h \Vert_{H^2(I)^d}
  (1 + \Vert u \Vert_{H^4(I)^d}^2)$.
  Also, since $v_h \in \ker B_{\widetilde{u}_h}$, we have that 
  $\I_{h,2}(v_h' \cdot \widetilde{u}_h') = 0$. As $v^h(a) = v_h(a)$, 
  we can use a  Poincaré inequality  on $v_h - v^h$ and obtain
  \begin{equation}\label{eq:projectionEstimate}
    \begin{aligned}
      \Vert v_h - v^h \Vert_{H^2(I)^d}
      &\leq c \Vert (v_h' \cdot u') u' \Vert_{H^1(I)^d}
      = c \Vert ((v_h' \cdot u') 
      - \I_{h,2} (v_h' \cdot \widetilde{u}_h') )u' \Vert_{H^1(I)^d} \\
      &\leq c \Vert ((v_h' \cdot u') - (v_h' \cdot \widetilde{u}_h'))u' \Vert_{H^1(I)^d}
      + c \Vert ((v_h' \cdot \widetilde{u}_h') 
      - \I_{h,2} (v_h' \cdot \widetilde{u}_h') )u' \Vert_{H^1(I)^d} \\
      &\leq c h^2 \Vert u \Vert_{H^4(I)^d}^2 \Vert v_h \Vert_{H^2(I)^d}^2
      + c h^2 \Vert u \Vert_{H^4(I)^d} \Vert D_h^3(v_h' \cdot \widetilde{u}_h')\Vert \\
      &\leq c h \Vert u \Vert_{H^4(I)^d}^2 \Vert v_h \Vert_{H^2(I)^d}.
    \end{aligned}
  \end{equation}
  This yields a bound on the first term:
  \begin{equation*}
    \label{eq:coerc_first_term}
    \alpha \Vert v^h \Vert_{H^2(I)^d}^2 
    \geq \alpha (1 - c_u h) \Vert v_h \Vert_{H^2(I)^d}^2.
  \end{equation*}
  For the second term we estimate
  \begin{align*}
    \vert a_{\widetilde{\lambda}_h}(v_h, v_h) - a_\lambda(v^h, v^h) \vert
    &\leq \vert (v_h'', v_h'') - ((v^h)'', (v^h)'') \vert
    + \vert (\widetilde{\lambda}_h, v_h' \cdot v_h')_{h,2}
    - (\lambda, (v^h)' \cdot (v^h)') \vert \\
    &= \Rom{4} + \Rom{5}.
  \end{align*}
   Term $\Rom{4}$  can easily be bound using the above estimates:
  \begin{equation*}\label{eq:IV}
    \Rom{4} = \vert (v_h'' - (v^h)'', v_h'' + (v^h)'') \vert
    \leq c h \Vert u \Vert_{H^4(I)^d}^2 \Vert v_h \Vert_{H^2(I)^d}^2
    (1 + \Vert u \Vert_{H^4(I)^d}^2).
  \end{equation*}
  For $\Rom{5}$, since $v_h \in \ker B_{\widetilde{u}_h} \subset \S^{3,1}_D(\T_h)^d$, we use again that $(\widetilde{\lambda}_h, v_h' \cdot v_h')_{h,2} = (\widetilde{\lambda}, v_h'
  \cdot v_h')_{h,2}$ for $\widetilde{\lambda} = \I_{h,2} \lambda$. Therefore we have
  \begin{align*}
    \Rom{5} 
    &\leq \vert (\widetilde{\lambda}, v_h' \cdot v_h')_{h,2}
    - (\widetilde{\lambda}, v_h' \cdot v_h') \vert
    + \vert (\widetilde{\lambda}, v_h' \cdot v_h') 
    - (\widetilde{\lambda}, (v^h)' \cdot (v^h)' ) \vert \\
    & \kurz+ \vert (\widetilde{\lambda}, (v^h)' \cdot (v^h)')
    - (\lambda, (v^h)' \cdot (v^h)' ) \vert
    =: \Rom{5}_1 + \Rom{5}_2 + \Rom{5}_3.
  \end{align*}
  Lemma \ref{lem:quadControl} and an inverse estimate imply
  \begin{align*}
    \Rom{5}_1
    &\leq c h^3 \int_I \vert D_h^3 (\widetilde{\lambda} v_h' \cdot v_h') \vert \ dx
    \leq c h \Vert \lambda \Vert_{H^2(I)} \Vert v_h \Vert_{H^2(I)^d}^2.
  \end{align*}
  For $\Rom{5}_2$ we can use (\ref{eq:projectionEstimate}) with the $H^1$ norm to obtain
  \begin{align*}
    \Rom{5}_2 
    \leq \Vert \widetilde{\lambda} \Vert \Vert v_h - v^h \Vert_{H^1(I)^d} 
    \Vert v_h + v^h \Vert_{H^2(I)^d}
    \leq c h^2 \Vert \lambda \Vert_{H^2(I)} \Vert u \Vert_{H^4(I)^d}^2 
    \Vert v_h \Vert_{H^2(I)^d}^2 (1 + \Vert u \Vert_{H^4(I)^d}^2).
  \end{align*}
  For $\Rom{5}_3$ we use an interpolation estimate to obtain
  \begin{align*}
    \Rom{5}_3 \leq \Vert \widetilde{\lambda} - \lambda \Vert
    \Vert (v^h)' \cdot (v^h)' \Vert
    \leq c h^2 \Vert \lambda \Vert_{H^2(I)} \Vert v_h \Vert_{H^2(I)^d}^2.
  \end{align*}
  Combined these estimates yield 
  $\Rom{5} \leq c_{u,\lambda} h \Vert v_h \Vert_{H^2(I)^d}^2$ and therefore
  \begin{align*}
    a_{\widetilde{\lambda}_h}(v_h, v_h)
    \geq \alpha (1 - c_u h) \Vert v_h \Vert_{H^2(I)^d}^2
    - c_{u,\lambda} h \Vert v_h \Vert_{H^2(I)^d}^2
    \geq \frac{\alpha}{2} \Vert v_h \Vert_{H^2(I)^d}^2
  \end{align*}
  for $h$ small enough.
\end{proof}

\begin{remark}
  The coercivity of $a_{\widetilde \lambda_h}$ on $\ker B_{\widetilde u_h}$ is another point where the $\P_1$ discretization of the inextensibility constraint can fail. If instead of $\I_{h,2}(\widetilde u_h' \cdot v_h') = 0$ we only  require  $\I_{h,1}(\widetilde u_h' \cdot v_h') = 0$ in $(\ref{eq:projectionEstimate})$ we get
  \begin{align*}
    \Vert v_h - v^h \Vert_{H^2(I)^d}
    &\leq c \Vert (v_h' \cdot (u' - \widetilde u_h')) u' \Vert_{H^1(I)^d}
      + c \Vert ((v_h' \cdot \widetilde u_h') - \I_{h,1}(v_h' \cdot \widetilde u_h')) u' \Vert_{H^1(I)^d} \\
    &\leq c h^2 \Vert u \Vert_{H^4(I)^d}^2 \Vert v_h \Vert_{H^2(I)^d}
      + c h \Vert u \Vert_{H^4(I)^d}^2 \Vert D_h^3 v_h \Vert
    \leq c \Vert u \Vert_{H^4(I)^d}^2 \Vert v_h \Vert_{H^2(I)^d},
  \end{align*}
  where in the last estimate we have used an inverse estimate.
  This ultimately yields
  \begin{align*}
    a_{\widetilde{\lambda}_h}(v_h, v_h)
    \geq \alpha (1 - c_u) \Vert v_h \Vert_{H^2(I)^d}^2
    - c_{u,\lambda} \Vert v_h \Vert_{H^2(I)^d}^2 ,
  \end{align*}
  which does not imply coercivity for $h$ small enough.
\end{remark}

In the proof of Lemma \ref{lem:coercivity} we have shown coercivity of
$a_{\widetilde{\lambda}_h}$ on $\ker B_{\widetilde{u}_h}$ by assuming that
the continuous bilinear form $a_{\lambda}$ is already coercive on
$\ker B_u$. We now discuss some settings where this assumption holds. \\
The coercivity assumption obviously holds for straight curves, since in this case we have $\lambda = - \vert u'' \vert^2 = 0$. 
Another important class of curves for which the coercivity assumption holds
are planar curves, i.e. curves $u: I \to \real^2$, as the following
proposition shows.

\begin{proposition}
  Let $I = [a,b] \subset \real$ an interval and  $u \in W^{2,\infty}(I)^2$  a curve
  satisfying $\vert u' \vert^2 = 1$. Further set $\lambda := - \vert u'' 
  \vert^2$ and define $G_u := \{v \in H^2(I)^d : u' \cdot v' = 0, v(a) = v'(a) = 0 \}  \supset \ker B_u $. Then $a_\lambda$ is
  coercive on $G_u$, i.e. there exists $\alpha > 0$ such that
  $a_\lambda(v,v) \geq \alpha \Vert v \Vert_{H^2(I)^d}^2$.
\end{proposition}

\begin{proof}
  Let $v \in G_u$ arbitrary. Since $u' \in \S^1$ and $u' \cdot v' = 0$, we
  can rewrite $v' = \gamma u_\perp'$ where $u_\perp'$ is a rotation of $u'$ by
  $\pi/2$ and $\gamma \in H^1(I)$ with $\gamma(a) = 0$. This yields
  \begin{align*}
    \vert v'' \vert^2
    &= \vert (\gamma u_\perp')' \vert^2
    = \vert \gamma' \vert^2 \vert u_\perp' \vert^2
    + \vert \gamma \vert^2 \vert u_\perp'' \vert^2 
    = \vert \gamma' \vert^2 + \vert \gamma \vert^2 \vert u'' \vert^2.
  \end{align*}
  
  From Hölder's inequality and a Poincaré inequality we thus obtain
  \begin{align*}
    \Vert v'' \Vert^2 = \int_I \vert \gamma' \vert^2 + \vert \gamma \vert^2 \vert u'' \vert^2 \ dx
    \leq \Vert \gamma' \Vert^2 + \Vert u'' \Vert_{L^\infty(I)}^2 \Vert \gamma \Vert^2 \leq (1 + c_P^2 \Vert u'' \Vert_{L^\infty(I)}^2) \Vert \gamma' \Vert^2 .
  \end{align*}
  
  Together with $\vert v' \vert^2 = \vert \gamma \vert^2$ this implies
  \begin{align*}
    a_\lambda(v,v) 
    &= \int_I \vert v'' \vert^2 - \vert u'' \vert^2 \vert v' \vert^2 \ dx
    = \int_I \vert \gamma' \vert^2 \ dx
    \geq (1 + c_P^2 \Vert u'' \Vert_{L^\infty(I)^d}^2)^{-1} 
    \Vert v'' \Vert_{L^2(I)^d}^2.
  \end{align*}
  Using another Poincaré inequality twice on $\Vert v \Vert_{H^2(I)^d}^2$
  therefore yields
  \begin{align*}
    \Vert v \Vert_{H^2(I)^d}^2
    \leq (1 + c_P^2 + c_P^4) \Vert v'' \Vert_{L^2(I)^d}^2.
  \end{align*}
  Thus the asserted inequality holds with $\alpha = (1 + c_P^2 \Vert u''
  \Vert_{L^\infty(I)^d}^2)^{-1} (1 + c_P^2 + c_P^4)^{-1}$.
\end{proof}

Next we have to show that $b_{\widetilde{u}_h}$ satisfies an inf-sup-condition.
We do this by constructing a function $v_h \in \S^{2,0}(\T_h)$ that satisfies
the second estimate from Lemma \ref{lem:Brezzi}. To construct this function we first
need some projections onto the finite-element-space.
Let $\widetilde{\Pi}_h: L^2(I) \to \S^{2,0}_0(\T_h)$ via
\begin{equation*}
  \widetilde{\Pi}_h v = \sum_{z \in \N_2 \setminus \Gamma_D'} 
  \frac{(v,\varphi_{z,2})_{L^2(I)^d}}{\beta_{z,2}}
  \varphi_{z,2}, \leer \beta_{z,2} = \int_I \varphi_{z,2} \ dx.
\end{equation*}
Here $\varphi_{z,2}$ is the nodal basis of $\S^{2,0}(\T_h)$.
Then we have that $(\widetilde{\Pi}_h v, w_h)_{h,2} = (v, w_h)$
 for all  $w_h \in \S^{2,0}_0(\T_h)$  and  $v \in L^2(I)$. 
Let $\Pi_h: L^2(I) \to \S^{2,0}_0(\T_h)$ be the
standard $L^2$ projection satisfying $(\Pi_h v, w_h) = (v, w_h) 
\ \forall w_h \in \S^{2,0}_0(\T_h), v \in L^2(I)$.
\begin{lemma}\label{lem:L2projection}
  For all $v \in H^1_0(I)$ the modified projection $\widetilde{\Pi}_h$ satisfies
  \begin{equation*}
    \Vert \widetilde{\Pi}_h v \Vert_{H^1(I)} \leq c \Vert v \Vert_{H^1(I)}.
  \end{equation*}
\end{lemma}
\begin{proof}
  Define $\delta_h := \widetilde{\Pi}_h v - \Pi_h v$. Lemma \ref{lemma:norm_equivalence}
  therefore implies
  \begin{align*}
    c \Vert \delta_h \Vert^2 \leq \Vert \delta_h \Vert_{h,2}^2
    = (\delta_h, \widetilde{\Pi}_h v - \Pi_h v)_{h,2}
    = (\delta_h, \Pi_h v) - (\delta_h, \Pi_h v)_{h,2}.
  \end{align*}
  Lemma \ref{lem:quadControl} then yields
  \begin{align*}
    \Vert \delta_h \Vert^2
    &\leq c h^3 (\Vert \delta_h \Vert \Vert D_h^3 (\Pi_h v) \Vert
    + \Vert D_h \delta_h \Vert \Vert D_h^2 (\Pi_h v) \Vert
    + \Vert D_h^2 \delta_h \Vert \Vert D_h (\Pi_h v) \Vert) \\
    &\leq c h \Vert \delta_h \Vert \Vert D_h (\Pi_h v) \Vert.
  \end{align*}
  We therefore obtain $\Vert \delta_h \Vert \leq c h \Vert (\Pi_h v)' \Vert$.
  An inverse estimate yields $\Vert \delta_h' \Vert \leq c \Vert (\Pi_h v)' \Vert$.
  The asserted inequality then follows from the $H^1$-stability of the standard
  projection $\Pi_h$  on quasi-uniform meshes.
\end{proof}

Now we show that the second condition of Lemma \ref{lem:Brezzi} holds.

\begin{lemma}[Inf-Sup-Condition]\label{lem:infsup}
  There exists $\beta > 0$ such that for all $\mu_h \in \S^{2,0}_0(\T_h)$ there exists $v_h \in \S^{3,1}_D(\T_h)^d \setminus\{0\}$ such that
  \begin{align*}
    b_{\widetilde u_h}(\mu_h, v_h) \geq \beta \Vert \mu_h \Vert_{H^1_0(I)'} \Vert v_h \Vert_{H^2(I)^d}.
  \end{align*}
\end{lemma}

\begin{proof}
  Let $\mu_h \in \S^{2,0}_0(\T_h)$ be arbitrary. Since $\S^{2,0}_0(\T_h) \hookrightarrow 
  H^1_0(I)'$,
  the Hahn-Banach theorem implies
  \begin{align*}
    \exists \phi \in H^1_0(I):\ \Vert \phi \Vert_{H^1(I)} = 1, \ 
    (\mu_h, \phi) = \Vert \mu_h \Vert_{H^1_0(I)'}.
  \end{align*}
  We now define $v_h \in \S^{3,1}_D(\T_h)$ via
  \begin{align*}
    v_h(x) := \int_a^x \I_{h,2}  ((\widetilde{\Pi}_h \phi) \widetilde{u}_h')  \ dx.
  \end{align*}
  Testing with $v_h$ and $\mu_h$ yields
  \begin{align*}
    b_{\widetilde{u}_h}(v_h,\mu_h)
    &= (\mu_h, \widetilde{u}_h' \cdot v_h')_{h,2}
    = \sum_{z \in \N_2} \mu_h(z) \widetilde{\Pi}_h \phi(z) \widetilde{u}_h'(z) 
    \cdot \widetilde{u}_h'(z) \beta_{h,2} \\
    &= \Vert \mu_h \Vert_{H^1_0(I)'}
    - \sum_{i=1}^M \mu_h(m_i) \widetilde{\Pi}_h \phi(m_i) \beta_{m_i,2}
    (1 - \vert \widetilde{u}_h'(m_i) \vert^2).
  \end{align*}
  For the second term, using $1 = \vert u' \vert^2$ we get
  \begin{align*}
    &\sum_{i=1}^M \mu_h(m_i) \widetilde{\Pi}_h \phi(m_i) \beta_{m_i,2}
    (1 - \vert \widetilde{u}_h'(m_i) \vert^2)\\
    &\kurz \leq \sum_{i=1}^M \Vert \mu_h \Vert_{L^\infty(I)} 
    \Vert \widetilde{\Pi}_h \phi \Vert_{L^\infty(I)}
    \vert \beta_{m_i,2} \vert \Vert u' - \widetilde{u}_h'\Vert_{L^\infty(I)^d}
    \Vert u' + \widetilde{u}_h' \Vert_{L^\infty(I)^d}.
  \end{align*}
  With $\vert \beta_{m_i,2} \vert \leq h$, $\Vert \cdot \Vert_{L^\infty(I)}
  \leq c  h^{-3/2}  \Vert \cdot \Vert_{H^1_0(I)'}$ on $\S^{2,0}_0(\T_h)$ and 
  $\Vert \cdot \Vert_{L^\infty(I)^d} \leq c \Vert \cdot \Vert_{H^1(I)^d}$ we obtain
  \begin{align*}
    \sum_{i=1}^M \mu_h(m_i) \widetilde{\Pi}_h \phi(m_i) \beta_{m_i,2}
    (1 - \vert \widetilde{u}_h'(m_i) \vert^2)
    &  \leq c h^{3/2} \sum_{i=1}^M \Vert \mu_h \Vert_{H^1_0(I)'} 
    \Vert \widetilde{\Pi}_h \phi \Vert_{H^1(I)}
    \Vert u \Vert_{H^4(I)^d}^2 \\
    &\leq c  h^{1/2}  \Vert \mu_h \Vert_{H^1_0(I)'} 
    \Vert u \Vert_{H^4(I)^d}^2.
  \end{align*}
  Therefore we have
  \begin{equation}\label{eq:b_estimate}
    b_{\widetilde{u}_h}(\mu_h, v_h)
    \geq (1 - c h^{1/2} \Vert u \Vert_{H^4(I)^d}^2) 
    \Vert \mu_h \Vert_{H^1_0(I)'}
    \geq \frac{1}{2} \Vert \mu_h \Vert_{H^1_0(I)'}
  \end{equation}
  for $h$ small enough. On the other side for $v_h$ the estimate
  \begin{align*}
    \Vert v_h \Vert_{H^2(I)^d}
    &\leq c \Vert \I_{h,2}((\widetilde{\Pi}_h \phi) \widetilde{u}_h') \Vert_{H^1(I)^d}
    \leq c \Vert u \Vert_{H^4(I)^d}
  \end{align*}
  holds. This yields $c^{-1} \Vert v_h \Vert_{H^2(I)^d} \Vert u \Vert_{H^4(I)^d}^{-1}
  \leq 1$
  and inserting this estimate into (\ref{eq:b_estimate}) implies
  \begin{align*}
    \frac{1}{2} c^{-1} \Vert v_h \Vert_{H^2(I)^d} \Vert u \Vert_{H^4(I)^d}^{-1} 
    \Vert \mu_h \Vert_{H^1_0(I)'}
    \leq b_{\widetilde{u}_h}(\mu_h, v_h),
  \end{align*}
  which is the asserted estimate.
\end{proof}

It remains to show that $DF_h$ is Lipschitz continuous in a neighbourhood of $(\widetilde u_h, \widetilde \lambda_h)$. This is established trough the following lemma.

\begin{lemma}[Lipschitz estimate]\label{lem:LipschitzEstimate}
  For all $(u_{h,1}, \lambda_{h,1}), (u_{h,2}, \lambda_{h,2}) \in X_h$ we have
  \begin{equation*}
    \Vert DF_h(u_{h,1}, \lambda_{h,1})
    - DF_h(u_{h,2}, \lambda_{h,2}) \Vert_{L(X_h, X_h')}
    \leq c \Vert (u_{h,1} - u_{h,2}, \lambda_{h,1} - \lambda_{h,2}) \Vert_{X_h}.
  \end{equation*}
\end{lemma}
\begin{proof}
  Let $(v_h, \mu_h), (w_h, \eta_h) \in X_h$ arbitrary. We then have
  \begin{align*}
    &\vert (DF_h(u_{h,1}, \lambda_{h,1})
    - DF_h(u_{h,2}, \lambda_{h,2}))[(v_h,\mu_h),(w_h, \eta_h)] \vert \\
    &\kurz \leq \vert (\lambda_{h,1} - \lambda_{h,2}, v_h' \cdot w_h')_{h,2} \vert
    + \vert (\mu_h, (u_{h,1} - u_{h,2}) \cdot w_h)_{h,2} \vert
    + \vert (\eta_h, (u_{h,1} - u_{h,2}) \cdot v_h)_{h,2} \vert.
  \end{align*}
  With Lemma \ref{lem:quadControl} and inverse norm estimates we get for the first term
  \begin{align*}
    &\vert (\lambda_{h,1} - \lambda_{h,2}, v_h' \cdot w_h')_{h,2} \vert\\
    &\kurz \leq \vert (\lambda_{h,1} - \lambda_{h,2}, v_h' \cdot w_h') \vert
    + \vert (\lambda_{h,1} - \lambda_{h,2}, v_h' \cdot w_h')_{h,2}
    - (\lambda_{h,1} - \lambda_{h,2}, v_h' \cdot w_h') \vert \\
    &\kurz \leq \Vert \lambda_{h,1} - \lambda_{h,2} \Vert_{H^1_0(I)'}
    \Vert v_h \Vert_{H^2(I)^d} \Vert w_h \Vert_{H^2(I)^d}
    + c h^3 \Vert \lambda_{h,1} - \lambda_{h,2} \Vert 
    \Vert D_h^3(v_h' \cdot w_h') \Vert\\
    &\leer + c h^3 \Vert D_h(\lambda_{h,1} - \lambda_{h,2}) \Vert
    \Vert D_h^2( v_h' \cdot w_h' ) \Vert
    + c h^3 \Vert D_h^2( \lambda_{h,1} - \lambda_{h,2}) \Vert
    \Vert D_h( v_h' \cdot w_h') \Vert\\
    &\kurz \leq c \Vert \lambda_{h,1} - \lambda_{h,2} \Vert_{H^1_0(I)'}
    \Vert v_h \Vert_{H^2(I)^d} \Vert w_h \Vert_{H^2(I)^d}.
  \end{align*}
  Analogously we get for the second and third term
  \begin{align*}
    \vert (\mu_h, (u_{h,1} - u_{h,2}) \cdot w_h)_{h,2} \vert
    &\leq c \Vert \mu_h \Vert_{H^1_0(I)'} \Vert u_{h,1} - u_{h,2} \Vert_{H^2(I)^d}
    \Vert w_h \Vert_{H^2(I)^d}, \\
    \vert (\eta_h, (u_{h,1} - u_{h,2}) \cdot v_h)_{h,2} \vert
    &\leq c \Vert \eta_h \Vert_{H^1_0(I)'} \Vert u_{h,1} - u_{h,2} \Vert_{H^2(I)^d}
    \Vert v_h \Vert_{H^2(I)^d},
  \end{align*}
  which implies the asserted estimate.
\end{proof}

We are now in a position to apply the inverse function theorem.

\begin{theorem}[Error estimate]\label{theorem:error_estimate_stationary}
  Let $(u,\lambda) \in (H^4(I)^d \times H^2(I)) \cap X$ be a solution to $F(u,\lambda) = 0$  such that $a_\lambda$ is coercive on $\ker B_u$ .
  Then for $h > 0$ sufficiently small there exists $(u_h,\lambda_h) \in X_h$
  with $F_h(u_h,\lambda_h) = 0$ and a constant $c > 0$ such that
  \begin{equation*}
    \Vert u - u_h \Vert_{H^2(I)^d} \leq c h^2.
  \end{equation*}
\end{theorem}
\begin{proof}
  Set $\widetilde{u}_h := \I_{h,3} u$ and $\widetilde{\lambda}_h := \I_{h,2,0} \lambda$.
   We show  that the conditions of the inverse function
  theorem of Lemma \ref{lemma:inverseFunction} are satisfied.
  Lemma \ref{lem:boundedness} implies the boundedness of 
  $F_h(\widetilde{u}_h, \widetilde{\lambda}_h)$ with
  $\kappa \leq c h^2$. Lemma \ref{lem:infsup} implies an inf-sup condition on
  $b_{\widetilde{u}_h}$ and Lemma \ref{lem:coercivity} yields the coercivity of
  $a_{\widetilde{\lambda}_h}$ on the kernel of $B_h$. In addition an interpolation
  estimate for the nodal $\mathcal{P}_2$-interpolant and an inverse estimate imply
  the boundedness of $a_{\widetilde{\lambda}_h}$ and $b_{\widetilde{u}_h}$. 
  Thus Brezzi's theorem implies
  the inverse estimate of Lemma \ref{lemma:inverseFunction}. Finally the Lipschitz
  condition follows from Lemma \ref{lem:LipschitzEstimate}. Now the inverse function
  theorem implies the existence of $(u_h, \lambda_h) \in X_h$ with
  $F_h(u_h, \lambda_h) = 0$ and
  \begin{align*}
    \Vert (\widetilde{u}_h - u_h, \widetilde{\lambda}_h - \lambda_h) \Vert_{X_h}
    \leq \varepsilon = c \kappa = c h^2.
  \end{align*}
  With $\Vert \cdot \Vert_{X_h} = \Vert \cdot \Vert_X$ and the interpolation estimate
  \begin{align*}
    \Vert \widetilde{u}_h - u \Vert_{H^2(I)^d} &\leq c h^2  \Vert u \Vert_{H^4(I)^d}
  \end{align*}
  we obtain the asserted estimate.
\end{proof}

\section{Numerical Approximation}\label{section:numerical_approximation}

\subsection{Time stepping scheme}\label{subsection:time_discretization}
In this section we verify the error estimate from Theorem
\ref{theorem:error_estimate_stationary} through numerical experiments.
For this we use the time stepping schemes proposed in \cite{Bartels13} and \cite{BDS25-pre},
which are obtained through time discretization of the $L^2$ gradient flow of the bending energy $E$.\\
We thus first choose a time interval
$[0,T] = \bigcup_{n = 1}^N [t_{n-1}, t_n]$,
with $t_n = n\tau$ and time step size $\tau$. Let $Z^n \in \S^{3,1}(\T_h)^d$ the
calculated approximation of $z_h(t_n)$, where $z_h$ is a solution to the $L^2$ gradient flow of $E$ in $\S^{3,1}(\T_h)$ with respect to the discrete constraint $\I_{h,2}(\vert z_h' \vert^2 - 1) = 0$, given boundary conditions and initial value $Z^0 = \J_{h,3} z_0$. Note that the discrete constraint
$\I_{h,2}(\vert z_h' \vert^2 - 1)= 0$ for the semi-discrete scheme can be
equivalently written as
\begin{align*}
  0 = \I_{h,2}(\vert z_h'(0) \vert^2 - 1), \leer
  0 = \frac{1}{2}\frac{d}{dt} \I_{h,2}(\vert z_h' \vert^2 - 1)
  = \I_{h,2}( \del_t z_h' \cdot z_h').
\end{align*}
We now linearize this constraint with respect to the previous time step and obtain the linearized discrete constraint
\begin{align*}
  0 = \I_{h,2}(\vert (Z^0)' \vert^2 - 1), \leer
  0 = \I_{h,2}((d_t Z^{n+1})' \cdot (Z^n)')
\end{align*}
for all $n \in \{0,...,N-1\}$
where $d_t$ is the backward difference quotient.
By also replacing the time derivative in the semi-discrete scheme with the 
backward difference quotient we obtain the fully discrete scheme:

\begin{algorithm}\label{equation:discrete_scheme}
  Choose $z_0 \in H^2(I)^d$ with $\vert z_0' \vert^2 = 1$, $z_0(a) = u_D(a)$ and $z_0' = u_D'$ on $\del I$.
  \begin{enumerate}
    \item Set $n = 0$ and
    $$ Z^0 := \J_{h,3} z_0 = z_0(a) + \int_a^x \I_{h,2} z_0' \ d\sigma. $$
    \item Find $d_t Z^{n+1} \in \G_h(Z^n) := \{Y \in \S^{3,1}_D(\T_h)^d~:~\I_{h,2}(Y' \cdot (Z^n)') = 0\}$ such that
    \begin{equation*}
      (d_t Z^{n+1}, Y) + \tau((d_t Z^{n+1})'', Y'') = - ((Z^n)'', Y'')
    \end{equation*}
    for all $Y \in \G_h(Z^n)$ and set $Z^{n+1} = Z^n + \tau d_t Z^{n+1}$.
    \item Stop if $n = N-1$, otherwise increase $n \to n+1$ and continue with $(2)$.
  \end{enumerate}
\end{algorithm}

Since the discretized, linearized
constraint defines a closed subspace of $\S^{3,1}(\T_h)^d$, the existence of 
discrete solutions follows immediately from the Lax-Milgram-Lemma.
Also, as in \cite{Bartels13}, we can incorporate boundary conditions
into the scheme by enforcing
\begin{itemize}
  \item $d_t Z^{n+1} = 0 \text{ on } \Gamma_D $, $ (d_t Z^{n+1})' = 0 \text{ on } 
  \Gamma_D' \kurz$ for fixed/clamped boundary conditions,
  \item $d_t Z^{n+1}(a) = d_t Z^{n+1}(b),\ 
  (d_t Z^{n+1})'(a) = (d_t Z^{n+1})'(b) \kurz $ for periodic boundary conditions.
\end{itemize}
Note that while we have $\I_{h,3}v(b) = v(b)$ for all $v \in C^1(I)^d$, this is in general not true for the interpolant $\J_{h,3}$. Therefore, the approximations $Z^n$ will in general only satisfy the boundary condition $Z^n(b) = Z^0(b)$ and not $Z^n(b) = u_D(b)$. This however only leads to a small error, since according to Lemma \ref{lemma:interpolation_estimate} we have $\vert Z^0(b) - u_D(b) \vert \leq c h^4$.
With $M$ the mass matrix and $S$ the fourth order stiffness matrix we therefore
have to solve in every time step the linear system of equations
\begin{equation}
  \label{equation:implementation scheme}
  \begin{bmatrix}
    M + \tau S & B^T \\ B & 0
  \end{bmatrix}
  \begin{bmatrix}
    d_t Z^{n+1} \\ \Lambda
  \end{bmatrix}
  = \begin{bmatrix}
    - SZ^n \\ 0
  \end{bmatrix},
\end{equation}
where $B$ is a matrix, that encodes the linearized discrete constraint
$\I_{h,2}((d_t Z^{n+1})' \cdot (Z^n)') = 0$ as well as the boundary conditions,
and $\Lambda$ is the unknown discrete Lagrange multiplier.

\begin{remark}
  \label{remark:H2 gradient flow}
  The mass matrix $M$ in $(\ref{equation:implementation scheme})$ can be replaced with the matrix $S$, yielding the new linear system of equations
  \begin{equation}
    \label{equation:implementation h2 scheme}
    \begin{bmatrix}
      (1 + \tau) S & B^T \\ B & 0
    \end{bmatrix}
    \begin{bmatrix}
      d_t Z^{n+1} \\ \Lambda
    \end{bmatrix}
    = \begin{bmatrix}
      - SZ^n \\ 0
    \end{bmatrix}.
  \end{equation}
  For sufficient boundary conditions this provides a discretization of the $H^2_D$ gradient flow of $E$ which in practice converges  faster  towards a minimizer than the $L^2$ gradient flow. If not enough boundary conditions are provided however, the bilinear form given by $S$ is no longer positive definite and the  matrix  in $(\ref{equation:implementation h2 scheme})$ becomes singular. The scheme $(\ref{equation:implementation scheme})$ can be used to approximate the $L^2$ gradient flow of $E$ even in the absence of  essential  boundary conditions.
\end{remark}

\subsection{Numerical experiments}

To experimentally verify the error estimate from Theorem \ref{theorem:error_estimate_stationary}  we choose starting values and boundary conditions for which the (locally) minimizing function is known. Let $u$ denote that local minimizer and $Z^n$ the calculated numerical approximation from scheme $(\ref{equation:discrete_scheme})$. We set $e_h^n := u - Z^n$ and apply the binomial theorem and Lemma \ref{lemma:rhs_interpolation} to obtain
\begin{equation*}
  \label{equation:H2_error_formula}
  \vert e_h^n \vert_{H^2(I)^d}^2
  = \vert u \vert_{H^2(I)^d}^2
  + \vert Z^n \vert_{H^2(I)^d}^2
  - 2 \int_I (\I_{h,3} u)'' \cdot (Z^n)'' \ dx.
\end{equation*}
We then set $\vert e_h \vert_{H^2} := \vert e_h^N \vert_{H^2(I)^d}$.
To also be able to estimate the weaker norms we additionally set $\widetilde e_h^n := \I_{h,3} e_h^n$ and $\vert \widetilde e_h \vert_{H^1} := \vert \widetilde e_h^N \vert_{H^1(I)^d}$, $\Vert \widetilde e_h \Vert_{L^2} := \Vert \widetilde e_h^N \Vert$.
\begin{example}[Semi-clamped circle]\label{experiment:circle}
  We choose $I = [0, 2\pi ]$ and $z_0(x) := (\cos(x), \sin(x))$. 
  Additionally we choose $T = 50$ and boundary conditions $u(0) = (1,0)$, $u'(0) = u'(2\pi) = (0,1)$.  
  Since $z_0$ is a local minimum for the bending energy, it is also the solution $u$ we want to approximate.  
  We now calculate solutions using both, the scheme $(\ref{equation:discrete_scheme})$ with $\P_2$ constraint as well as the scheme with $\P_1$ constraint from \cite{Bartels13}. The results for the $H^2$ approximation error are displayed in Table \ref{table:circle_p3}. We observe linear convergence for the $\P_1$ constraint and quadratic convergence for the $\P_2$ constraint, as predicted by Theorem \ref{theorem:error_estimate_stationary}.
  \ 
  \begin{filecontents*}{tables/circle_open_end_h2.csv}
1.57080,1.290e+00,-,1.340e+00,-,1.371e+00,-,2.228e-01,-,2.228e-01,-,2.228e-01,-
0.78540,5.628e-01,1.19662,5.629e-01,1.25179,5.629e-01,1.28421,5.714e-02,1.96322,5.714e-02,1.96322,5.714e-02,1.96322
0.39270,2.834e-01,0.98972,2.834e-01,0.98979,2.834e-01,0.98984,1.438e-02,1.99081,1.438e-02,1.99081,1.438e-02,1.99081
0.19635,1.420e-01,0.99724,1.420e-01,0.99726,1.420e-01,0.99727,3.600e-03,1.99770,3.600e-03,1.99770,3.600e-03,1.99770
0.09817,7.103e-02,0.99930,7.103e-02,0.99930,7.103e-02,0.99931,9.003e-04,1.99939,9.003e-04,1.99939,9.003e-04,1.99939
\end{filecontents*}
\DTLloaddb[noheader,keys={h,e11,r11,e12,r12,e13,r13,e21,r21,e22,r22,e23,r23}]
  {circle_h2_DB}
  {tables/circle_open_end_h2.csv}  
  \begin{center}\begin{minipage}{.9\linewidth}
      \resizebox{\linewidth}{!}{
        \begin{tabular}{c|cc|cc|cc|cc}\toprule
          \multirow{3}{*}{$h$}
          & \multicolumn{4}{c|}{$\mathcal{P}_1$ constraint}
          & \multicolumn{4}{c}{$\mathcal{P}_2$ constraint} \\
          \cmidrule{2-9}
          & \multicolumn{2}{c|}{$\tau = 1/10$} 
          & \multicolumn{2}{c|}{$\tau = 1/20$}
          & \multicolumn{2}{c|}{$\tau = 1/10$}
          & \multicolumn{2}{c}{$\tau = 1/20$}\\
          & $\vert e_{h} \vert_{H^2}$ & eoc 
          & $\vert e_{h} \vert_{H^2}$ & eoc
          & $\vert e_{h} \vert_{H^2}$ & eoc
          & $\vert e_{h} \vert_{H^2}$ & eoc
          \DTLforeach*{circle_h2_DB}
          {\h=h,\a=e11,\b=r11,\c=e12,\d=r12,\e=e21,\f=r21,\g=e22,\i=r22}{%
            \DTLiffirstrow{\\\cmidrule{1-9}}{\\}%
            \h & \a & \b & \c & \d & \e & \f & \g & \i
          }%
          \\\bottomrule
      \end{tabular}}   
      \vspace{-.33cm}
      \captionof{table}
      {
        \label{table:circle_p3}
        $H^2$ approximation error for the schemes with $\P_1$ and $\P_2$ 
        constraint and starting value $Z^0$ in Example \ref{experiment:circle}.
        The approximation error converges linearly in case of the $\P_1$ constraint and quadratically in case of the $\P_2$ constraint. \\
      } 
  \end{minipage}\end{center}
  According to Corollary \ref{corollary:P1_constraint_minimizer}, the suboptimal
  linear convergence of the solutions to the $\P_1$ scheme stems from
  these solutions being piecewise $\P_2$ instead of $\P_3$.
  To verify this result experimentally we set
  $$
  \widetilde Z^0(x) := z_0(0) + \int_0^x \I_{h,1}(z_0') \ dx \in \S^{2,1}(\T_h)^d
  $$
  and insert it as initial value into the discrete scheme using the $\P_1$ constraint.  By comparing the stationary values with the starting values we observe that $\widetilde Z^0$ is stationary for $E$ under the $\P_1$ constraint, meaning the solution is piecewise quadratic, while $Z^0$ is stationary for $E$ under the $\P_2$ constraint, meaning the solution in this case is piecewise cubic.
\end{example}

\begin{figure}
  \begin{center}
    \begin{minipage}{.3\linewidth}
      \includegraphics[trim={4cm 15mm 5cm 15mm},clip,width=\textwidth]
      {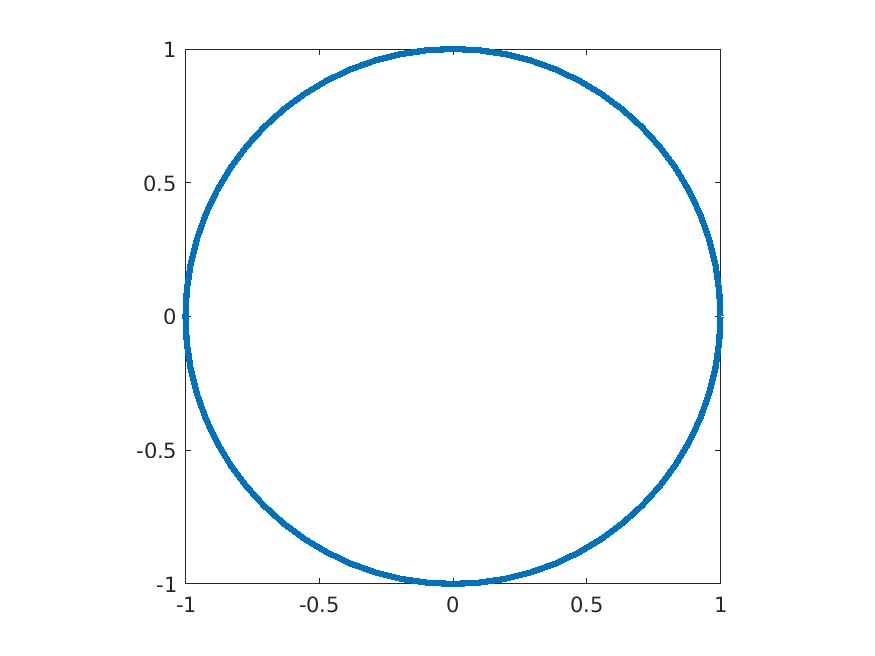}
    \end{minipage}
    \begin{minipage}{.3\linewidth}
      \includegraphics[trim={5cm 1cm 5cm 4cm},clip,width=\textwidth]
      {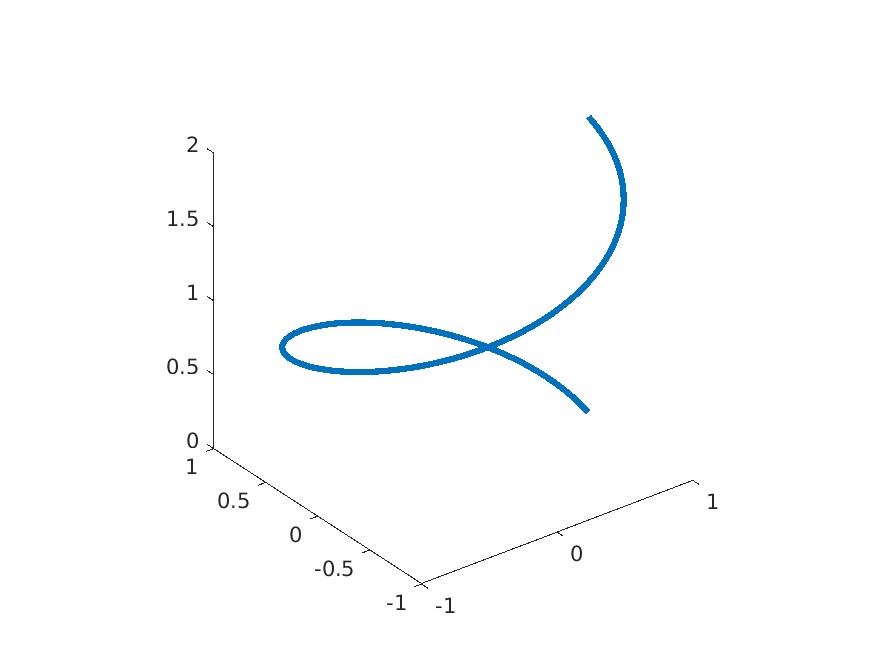}
    \end{minipage}
    \begin{minipage}{.3\linewidth}
      \includegraphics[trim={5cm 1cm 6cm 2cm},clip,width=\textwidth]
      {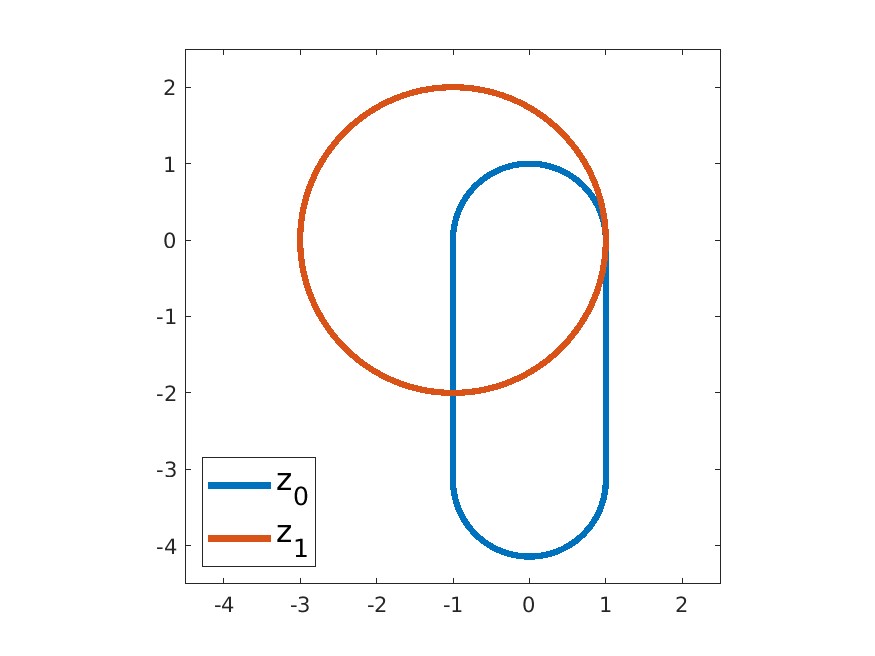}
    \end{minipage}  
    \begin{minipage}{.9\linewidth}
      \caption
      {
        Initial and stationary functions for Example \ref{experiment:circle} (left), Example \ref{experiment:helix} (center) and Example \ref{experiment:oval} (right). 
      }
    \end{minipage}
  \end{center}
\end{figure}

\begin{example}[Clamped helix]\label{experiment:helix}
   We next look at  a curve in three-dimensional space.
  We choose $I = [0, 2\sqrt{\pi^2 + 1}]$, $\lambda = \pi / \sqrt{\pi^2 + 1}$,
  $\mu = 1 / \sqrt{\pi^2 + 1}$ and define $z: I \times [0,50] \to \real^3$, 
  $$
  z_0(x) := (\cos(\lambda x), \sin(\lambda x), \mu x).
  $$ 
  This curve describes
  a helix and for clamped boundary conditions, i.e. $\Gamma_D = \Gamma_D' = 
  \del I$, $z(t)$ is minimal for the 
  bending energy for all $t$ and thus a solution to the elastic flow.
  We again compute $\vert e_h \vert_{H^2}$ and $\widetilde e_h$ in the weaker norms for the schemes with $\mathcal{P}_1$ and $\mathcal{P}_2$ discretization of the constraint. The results are displayed in  Tables \ref{table:helix H2 error} and \ref{table:helix weak norms}.  The results  nearly  coincide with the results that we obtained from Example \ref{experiment:circle}, even though we now used clamped boundary conditions that are not covered by Theorem \ref{theorem:error_estimate_stationary}, suggesting that the error estimate can be extended to cover clamped and fixed boundary conditions as well. In case of the $\P_1$ constraint we observe linear convergence in $H^2$ and quadratic convergence in the weaker norms. In case of the $\P_2$ constraint the rate of convergence is again the same as for the interpolation error and thus quasi-optimal.
  \\
  \begin{filecontents*}{tables/helix_h2_combined.csv}
1.64845,1.070e+00,-,1.070e+00,-,1.070e+00,-,2.081e-01,-,2.081e-01,-,2.081e-01,-
0.82423,5.498e-01,0.96008,5.498e-01,0.96008,5.498e-01,0.96008,5.320e-02,1.96792,5.320e-02,1.96792,5.320e-02,1.96792
0.41211,2.768e-01,0.99030,2.768e-01,0.99031,2.768e-01,0.99031,1.338e-02,1.99194,1.338e-02,1.99194,1.338e-02,1.99194
0.20606,1.386e-01,0.99759,1.386e-01,0.99759,1.386e-01,0.99759,3.348e-03,1.99798,3.348e-03,1.99798,3.348e-03,1.99798
0.10303,6.934e-02,0.99940,6.934e-02,0.99940,6.934e-02,0.99940,8.374e-04,1.99943,8.375e-04,1.99941,8.375e-04,1.99942
\end{filecontents*}
\DTLloaddb[noheader,keys={h,e11,r11,e12,r12,e13,r13,e21,r21,e22,r22,e23,r23}]
  {helix_h2_DB}
  {tables/helix_h2_combined.csv}  
  \begin{center}\begin{minipage}{.9\linewidth}
      \resizebox{\linewidth}{!}{
        \begin{tabular}{c|cc|cc|cc|cc}\toprule
          \multirow{3}{*}{$h$}
          & \multicolumn{4}{c|}{$\mathcal{P}_1$ constraint}
          & \multicolumn{4}{c}{$\mathcal{P}_2$ constraint} \\
          \cmidrule{2-9}
          & \multicolumn{2}{c|}{$\tau = 1/10$} 
          & \multicolumn{2}{c|}{$\tau = 1/20$}
          & \multicolumn{2}{c|}{$\tau = 1/10$}
          & \multicolumn{2}{c}{$\tau = 1/20$}\\
          & $\vert e_{h} \vert_{H^2}$ & eoc 
          & $\vert e_{h} \vert_{H^2}$ & eoc
          & $\vert e_{h} \vert_{H^2}$ & eoc
          & $\vert e_{h} \vert_{H^2}$ & eoc
          \DTLforeach*{helix_h2_DB}
          {\h=h,\a=e11,\b=r11,\c=e12,\d=r12,\e=e21,\f=r21,\g=e22,\i=r22}{%
            \DTLiffirstrow{\\\cmidrule{1-9}}{\\}%
            \h & \a & \b & \c & \d & \e & \f & \g & \i
          }%
          \\\bottomrule
      \end{tabular}}
      \vspace{-3.3mm}
      \captionof{table}
      {
        \label{table:helix H2 error}
        Approximation error in $H^2(I)^d$ for the $\P_1$ and $\P_2$ constraint and initial value $Z^0$ in Example \ref{experiment:helix}. We observe linear convergence for the $\P_1$ constraint and quadratic convergence for the $\P_2$ constraint. \\
      }
  \end{minipage}\end{center}
  
  \begin{filecontents*}{tables/helix_weak_norms_stationary.csv}
1.64845,7.999e-01,-,5.644e-01,-,9.344e-03,-,8.179e-03,-
0.82423,2.029e-01,1.97861,1.491e-01,1.92025,5.620e-04,4.05530,5.095e-04,4.00467
0.41211,5.108e-02,1.99016,3.782e-02,1.97909,3.497e-05,4.00668,3.184e-05,4.00050
0.20606,1.280e-02,1.99659,9.497e-03,1.99376,2.183e-06,4.00131,1.990e-06,4.00011
0.10303,3.203e-03,1.99896,2.377e-03,1.99825,1.364e-07,4.00032,1.243e-07,4.00003
\end{filecontents*}
\DTLloaddb[noheader,keys={h,e11,r11,e12,r12,e21,r21,e22,r22}]
  {helix_weak_norms_DB}
  {tables/helix_weak_norms_stationary.csv}  
  \begin{center}\begin{minipage}{.9\linewidth}
      \resizebox{\linewidth}{!}{
        \begin{tabular}{c|cc|cc|cc|cc}\toprule
          \multirow{2}{*}{$h$}
          & \multicolumn{4}{c|}{$\mathcal{P}_1$ constraint}
          & \multicolumn{4}{c}{$\mathcal{P}_2$ constraint} \\
          \cmidrule{2-9}
          & $\Vert \widetilde e_{h} \Vert_{L^2}$ & eoc 
          & $\vert \widetilde e_{h} \vert_{H^1}$ & eoc
          & $\Vert \widetilde e_{h} \Vert_{L^2}$ & eoc
          & $\vert \widetilde e_{h} \vert_{H^1}$ & eoc
          \DTLforeach*{helix_weak_norms_DB}
          {\h=h,\a=e11,\b=r11,\c=e12,\d=r12,\e=e21,\f=r21,\g=e22,\i=r22}{%
            \DTLiffirstrow{\\\cmidrule{1-9}}{\\}%
            \h & \a & \b & \c & \d & \e & \f & \g & \i
          }%
          \\ \bottomrule
      \end{tabular}}
      \vspace{-3.3mm}
      \captionof{table}
      {
        \label{table:helix weak norms}
        Approximation error $\widetilde e_h$ in the weaker $L^2$ norm and $H^1$ semi-norm for the $\P_1$ and $\P_2$ constraint, initial value $Z^0$ and $\tau = 1/10$ in Example \ref{experiment:helix}. Quadratic convergence can be observed for the $\P_1$ constraint. For the quadratic constraint however we observe quartic convergence, which together with the interpolation estimates yields quartic and cubic convergence. \\
      }
  \end{minipage}\end{center}
\end{example}

\begin{example}
  [Semi-clamped oval]\label{experiment:oval}
  In this experiment we verify the results from Theorem
  \ref{theorem:error_estimate_stationary} with a non-stationary starting
  value. For this we set $I := [0,4\pi]$ and define $z_0, z_1: I \to \real^2$,
  \begin{align*}
    z_0(x) &:=
    \begin{cases}
      (\cos x, \sin x), \leer &x \in [0,\pi] \\
      (-1,\pi - x), &x \in [\pi,2\pi] \\
      (\cos(x-\pi),\sin(x-\pi) - \pi), &x \in [2\pi, 3\pi] \\
      (1, x - 4\pi), &x \in [3\pi, 4\pi]
    \end{cases}, \\
    z_1(x) &:=
    2 \left(\cos \frac{x}{2} - \frac{1}{2}, \sin \frac{x}{2} \right).
  \end{align*}
  We note that $z_1$ is a local minimizer of the bending energy under
  boundary conditions $z(0) = (1,0),\ z'(0) = (0,1) = z'(4\pi)$ and
  therefore a solution to the bending problem and with $z_0$ as starting
  value for the continuous $L^2$ gradient flow $z$ we get 
  $z(t) \xrightarrow{t \to \infty} z_1$. We again compute the $H^2$, $H^1$ and $L^2$ approximation errors for the $\P_1$ and $\P_2$ constraint. The results are displayed in Tables \ref{table:oval H2 error} and \ref{table:oval weak norms}. Again the observed results are consistent with the previous two examples. The convergence rate for the $\P_2$ constraint is quasi-optimal while for the $\P_1$ constraint we only observe linear convergence in $H^2$ and quadratic convergence in $H^1$ and $L^2$.
  
  \begin{filecontents*}{tables/oval_l2_T=5000_tau=0.001_combined.csv}
3.14159,1.775e+00,-,1.775e+00,-,1.774e+00,-,1.575e-01,-,1.575e-01,-,1.575e-01,-
1.57080,3.978e-01,2.15754,3.978e-01,2.15749,3.978e-01,2.15746,4.041e-02,1.96300,4.040e-02,1.96333,4.040e-02,1.96333
0.78540,2.005e-01,0.98840,2.004e-01,0.98871,2.004e-01,0.98883,1.186e-02,1.76883,1.079e-02,1.90445,1.037e-02,1.96252
\end{filecontents*}
\DTLloaddb[noheader,keys={h,e11,r11,e12,r12,e13,r13,e21,r21,e22,r22,e23,r23}]
  {oval_DB}
  {tables/oval_l2_T=5000_tau=0.001_combined.csv}  
  \begin{center}\begin{minipage}{.9\linewidth}
      \resizebox{\linewidth}{!}{
        \begin{tabular}{c|cc|cc|cc|cc}\toprule
          \multirow{3}{*}{$h$}
          & \multicolumn{4}{c|}{$\mathcal{P}_1$ constraint}
          & \multicolumn{4}{c}{$\mathcal{P}_2$ constraint} \\
          \cmidrule{2-9}
          & \multicolumn{2}{c|}{$\tau = 1/2000$} 
          & \multicolumn{2}{c|}{$\tau = 1/4000$}
          & \multicolumn{2}{c|}{$\tau = 1/2000$}
          & \multicolumn{2}{c}{$\tau = 1/4000$}\\
          & $\vert e_{h} \vert_{H^2}$ & eoc 
          & $\vert e_{h} \vert_{H^2}$ & eoc
          & $\vert e_{h} \vert_{H^2}$ & eoc
          & $\vert e_{h} \vert_{H^2}$ & eoc
          \DTLforeach*{oval_DB}
          {\h=h,\a=e12,\b=r12,\c=e13,\d=r13,\e=e22,\f=r22,\g=e23,\i=r23}{%
            \DTLiffirstrow{\\\cmidrule{1-9}}{\\}%
            \h & \a & \b & \c & \d & \e & \f & \g & \i
          }%
          \\\bottomrule
      \end{tabular}}
      \vspace{-3.3mm}
      \captionof{table}
      {
        \label{table:oval H2 error}
        Calculated $H^2$ approximation error for the discrete $\P_1$ and $\P_2$ constraint and starting value $Z^0$ in Example \ref{experiment:oval}.
        For the $\P_1$ constraint linear convergence is observed. For the $\P_2$ constraint we observe quadratic convergence. \\
      }
  \end{minipage}\end{center}
  
  \begin{filecontents*}{tables/oval_weak_norms_stationary.csv}
3.14159,2.944e+01,-,4.997e+00,-,2.142e-02,-,7.977e-03,-
1.57080,5.108e-01,5.84885,1.965e-01,4.66809,1.355e-03,3.98187,5.176e-04,3.94615
0.78540,1.285e-01,1.99089,4.970e-02,1.98347,1.194e-04,3.50492,4.390e-05,3.55955
\end{filecontents*}
\DTLloaddb[noheader,keys={h,e12,r12,e13,r13,e22,r22,e23,r23}]
  {oval_weak_norms_DB}
  {tables/oval_weak_norms_stationary.csv}  
  \begin{center}\begin{minipage}{.9\linewidth}
      \resizebox{\linewidth}{!}{
        \begin{tabular}{c|cc|cc|cc|cc}\toprule
          \multirow{2}{*}{$h$}
          & \multicolumn{4}{c|}{$\mathcal{P}_1$ constraint}
          & \multicolumn{4}{c}{$\mathcal{P}_2$ constraint} \\
          \cmidrule{2-9}
          & $\Vert \widetilde e_{h} \Vert_{L^2}$ & eoc 
          & $\vert \widetilde e_{h} \vert_{H^1}$ & eoc
          & $\Vert \widetilde e_{h} \Vert_{L^2}$ & eoc
          & $\vert \widetilde e_{h} \vert_{H^1}$ & eoc
          \DTLforeach*{oval_weak_norms_DB}
          {\h=h,\a=e12,\b=r12,\c=e13,\d=r13,\e=e22,\f=r22,\g=e23,\i=r23}{%
            \DTLiffirstrow{\\\cmidrule{1-9}}{\\}%
            \h & \a & \b & \c & \d & \e & \f & \g & \i
          }%
          \\\bottomrule
      \end{tabular}}
      \vspace{-3.3mm}
      \captionof{table}
      {
        \label{table:oval weak norms}
        Computed $L^2$ and $H^1$ approximation error for the discrete $\P_1$ and $\P_2$ constraint, time step size $\tau = 1/4000$ and initial value $Z^0$ in Example \ref{experiment:oval}.\\
      }
  \end{minipage}\end{center}
  The choice of $T = 5000$ and $\tau \in \{1/2000, 1/4000\}$ results in
  high computational effort and thus slow performance. It is however
  necessary due to the relatively slow convergence of the $L^2$ gradient
  flow. As mentioned in Remark \ref{remark:H2 gradient flow} this can be improved using the $H^2$ gradient flow instead.  To demonstrate this we run the discrete scheme $(\ref{equation:implementation scheme})$ and the discrete scheme $(\ref{equation:implementation h2 scheme})$ with $T = 50$ and initial value $Z^0$. The results are displayed in Tables \ref{table:oval H2 error H2 flow}. For the $H^2$ gradient flow the $H^2$ approximation error is about the same as above, but at $0.1 \%$ the computational cost. For the $L^2$ gradient flow we do not get a good approximation at all.
  
  \begin{filecontents*}{tables/oval_p2_stationary_flow_comparison.csv}
3.14159,9.239e-01,-,9.238e-01,-,1.575e-01,-,1.575e-01,-
1.57080,8.081e-01,0.19318,8.079e-01,0.19346,4.043e-02,1.96199,4.041e-02,1.96294
0.78540,8.037e-01,0.00791,8.035e-01,0.00792,1.032e-02,1.96943,1.020e-02,1.98530
\end{filecontents*}
\DTLloaddb[noheader,keys={h,e12,r12,e13,r13,e22,r22,e23,r23}]
  {oval_flow_comparison_DB}
  {tables/oval_p2_stationary_flow_comparison.csv}  
  \begin{center}\begin{minipage}{.9\linewidth}
      \resizebox{\linewidth}{!}{
        \begin{tabular}{c|cc|cc|cc|cc}\toprule
          \multirow{3}{*}{$h$}
          & \multicolumn{4}{c|}{$L^2$ gradient flow}
          & \multicolumn{4}{c}{$H^2$ gradient flow} \\
          \cmidrule{2-9}
          & \multicolumn{2}{c|}{$\tau = 1/200$} 
          & \multicolumn{2}{c|}{$\tau = 1/400$}
          & \multicolumn{2}{c|}{$\tau = 1/200$}
          & \multicolumn{2}{c}{$\tau = 1/400$}\\
          & $\vert e_{h} \vert_{H^2}$ & eoc 
          & $\vert e_{h} \vert_{H^2}$ & eoc
          & $\vert e_{h} \vert_{H^2}$ & eoc
          & $\vert e_{h} \vert_{H^2}$ & eoc
          \DTLforeach*{oval_flow_comparison_DB}
          {\h=h,\a=e12,\b=r12,\c=e13,\d=r13,\e=e22,\f=r22,\g=e23,\i=r23}{%
            \DTLiffirstrow{\\\cmidrule{1-9}}{\\}%
            \h & \a & \b & \c & \d & \e & \f & \g & \i
          }%
          \\\bottomrule
      \end{tabular}}
      \vspace{-3.3mm}
      \captionof{table}
      {
        \label{table:oval H2 error H2 flow}
        Calculated $H^2$ approximation error for the discrete $L^2$ and $H^2$ gradient flow with $\P_2$ constraint and starting value $Z^0$ in Example~\ref{experiment:oval}. \\
      }
  \end{minipage}\end{center}
\end{example}

\renewcommand{\thesection}{A}
\section{Appendix}

\begin{lemma}[Interpolation estimate]\label{lemma:interpolation_estimate}
  Let $I = \bigcup_{i = 1}^M [x_{i-1}, x_i]$ a decomposition of an interval $I$ with
  $\vert x_i - x_{i-1} \vert \leq h$ for all $i$. Further let $\I_{h,m}$ be the
  Lagrange interpolation operator of polynomial degree $m \in \{1,2,3\}$. 
  Then $\I_{h,m}$ satisfies
  \begin{equation*}
    \vert u - \I_{h,m} u \vert_{W^{k,p}(I)^d} 
    \leq c h^{r-k} \vert u \vert_{W^{r,p}(I)^d}
  \end{equation*}
  for all $k \in \{0,...,r\}$,
  where $r \in \{\max(1,m-1),...,m+1\}$ arbitrary.
  Further, for $k \geq 1$, the interpolant $\J_{h,3}$ satisfies the same estimate.
  For $k = 0$ the Simpson rule from Lemma \ref{lemma:simpson_rule}
  implies
  \begin{equation*}
    \Vert u - \J_{h,3} u \Vert_{L^\infty(I)}
    \leq c h^4 \vert u \vert_{W^{5,\infty}(I)}.
  \end{equation*}
\end{lemma}
\begin{proof}
  The estimate follows from using local estimates on each subinterval and summing
  up over all intervals. The local estimate used is a special case of
  \cite[Theorem 4.4.4]{BS08} that is obtained by using $\P_{r-1} \subset \P_m$.
\end{proof}

\begin{lemma}\label{lemma:norm_equivalence}
  There exists a constant $c > 0$ that is independent of $h$ such that for all $v_h \in \S^{2,0}(\T_h)^d$ we have
  $$
  c^{-1} \Vert v_h \Vert^2 
  \leq \Vert v_h \Vert_{h,2}^2 
  \leq \Vert \I_{h,2}(\vert v_h \vert^2) \Vert_{L^1(I)}
  \leq c \Vert v_h \Vert^2.
  $$
\end{lemma}

\begin{proof}
  To show the first estimate we note that $(\cdot,\cdot)_{h,2}$ is a scalar product on $\S^{2,0}(\T_h)^d$. It is easy to see that $(\cdot, \cdot)_{h,2}$ is symmetric and bilinear. The positive definiteness follows immediately from the Simpson rule:
  \begin{align*}
    \Vert v_h \Vert_{h,2}^2 
    = \int_{i=1}^M \int_{I_i} \I_{h,2}(\vert v_h \vert^2) \ dx
    = \int_{i=1}^M \frac{h_i}{6} (\vert v_h(x_{i-1}) \vert^2 + 4 \vert v_h(m_i) \vert^2 + \vert v_h(x_i) \vert^2) > 0
  \end{align*}
  for $v_h \in \S^{2,0}(\T_h)^d \setminus \{0\}$. Therefore $\Vert \cdot \Vert_{h,2}$ defines a norm on $\S^{2,0}(\T_h)^d$ and since $\S^{2,0}(\T_h)^d$ has finite dimension, all norms on $\S^{2,0}(\T_h)^d$ are equivalent. Let now $\I_2: C^0([0,1])^d \to \P_2^d$ denote the standard $\P_2$ interpolant on $[0,1]$. We define $\phi_i: [0,1] \to I_i$, $\phi_i(t) := x_{i-1} + t h_i$. Elementwise transformation onto $[0,1]$ then yields
  \begin{align*}
    \Vert v_h \vert^2 = \sum_{i=1}^M \int_{I_i} \vert v_h \vert^2 \ dx
    = \sum_{i=1}^M h_i \int_0^1 \vert v_h \circ \phi_i \vert^2 \ dx
    = \sum_{i=1}^M h_i \Vert v_h \circ \phi_i \Vert_{L^2([0,1])^d}^2.
  \end{align*}
  Since by construction we have $v_h \circ \phi_i \in \P_2^d$ we can now use the norm equivalence on $\P_2^d$ to obtain a constant $c > 0$ that is independent of $h$ such that
  \begin{align*}
    \Vert v_h \circ \phi_i \Vert_{L^2([0,1])^d}^2
    \leq c \int_0^1 \I_2(\vert v_h \circ \phi_i \vert^2) \ dx.
  \end{align*}
  Another application of the transformation theorem then yields
  \begin{align*}
    \Vert v_h \Vert^2
    \leq \sum_{i=1}^M c h_i \int_0^1 \I_2(\vert v_h \circ \phi_i \vert^2) \ dx
    = c \sum_{i=1}^M \int_{I_i} \I_{h,2}(\vert v_h \vert^2) \ dx
    = c \Vert v_h \Vert_{h,2}^2,
  \end{align*}
  which proves the first estimate.
  The second estimate is trivial. For the third estimate we again use the transformation theorem as well as the stability estimate from \cite[Lemma 4.4.1]{BS08} and basic norm equivalences to obtain:
  \begin{align*}
    \Vert \I_{h,2}(v_h^2) \Vert_{L^1(I)}
    &= \sum_{i=1}^M h_i \int_0^1 \vert \I_2 (v_h \circ \phi_i)^2 \vert \ dx
    \leq c \sum_{i=1}^M h_i \Vert v_h \circ \phi_i \Vert_{L^\infty([0,1])^d}^2 \\
    &\leq c \sum_{i=1}^M h_i \Vert v_h \circ \phi_i \Vert_{L^2([0,1])^d}^2
    = c \Vert v_h \Vert^2,
  \end{align*}
  which finishes the proof.
\end{proof}

\begin{lemma}\label{lemma:L1_estimate}
  Let $u \in C^0(I), v \in C^0(I)^d$ with $\vert v(z) \vert = 1$ for all
  $z \in \N_2(\T_h)$.
  Then we have
  $$
  \Vert \I_{h,2}(uv) \Vert_{L^1(I)^d} \leq c \Vert \I_{h,2}(u) \Vert_{L^1(I)}.
  $$
\end{lemma}

\begin{proof}
  We first set $I_0 = [0,1]$ and for $p \in \P_2(I_0)$ we set
  \begin{align*}
    \Vert p \Vert_{\max} := \max \{\vert p(0) \vert, \vert p(1/2) \vert, \vert p(1) \vert \}.
  \end{align*}
  It is easy to see that $\Vert \cdot \Vert_{\max}$ and $\Vert \cdot \Vert_{L^1(I)}$ define norms on $\P_2(I_0)$, and since $\P_2(I_0)$ has finite dimension, those two norms are equivalent. Now we define $\Phi_i: I_0 \to I_i$ by $\Phi_i(t) := (1-t)x_{i-1} + t x_i$. The transformation theorem then yields for all $p \in \P_2(I_i)$
  \begin{align*}
    \Vert p \Vert_{L^1(I_i)}
    = h_i \Vert p \circ \Phi_i \Vert_{L^1(I_0)}
    \cong h_i \max \{\vert p(x_{i-1}) \vert, \vert p(m_i) \vert, \vert p(x_i) \vert \}.
  \end{align*}
  We therefore obtain
  \begin{align*}
    \Vert \I_{h,2}(uv) \Vert_{L^1(I)^d}
    &= \sum_{i=1}^M \Vert \I_{h,2}(uv) \Vert_{L^1(I_i)^d}
    \cong \sum_{i=1}^M \max \{\vert (uv)(x_{i-1}) \vert, \vert (uv)(m_i) \vert, \vert (uv)(x_i) \vert \} \\
    &= \sum_{i=1}^M \max \{\vert u(x_{i-1}) \vert, \vert u(m_i) \vert, \vert u(x_i) \vert \}
    \cong \Vert \I_{h,2} u \Vert_{L^1(I)},
  \end{align*}
  which finishes the proof.
\end{proof}

\begin{lemma}[Simpson rule]\label{lemma:simpson_rule}
  On each subinterval $I_i = [x_{i-1}, x_i]$ the nodal interpolant $\I_{h,2}$ satisfies 
  \begin{equation*}
    \left\vert \int_{I_i} f \ dx - \int_{I_i} \I_{h,2} f \ dx \right\vert
    \leq C h_i^5 \max_{x \in I_i} \vert D_h^4 f(x) \vert.
  \end{equation*}
  Further we have
  \begin{equation*}
    \left\vert \int_{I} f \ dx - \int_{I} \I_{h,2} f \ dx \right\vert
    \leq \sum_{i=1}^M C h_i^5 \max_{x \in I_i} \vert D_h^4 f(x) \vert
    \leq C h^4 \Vert D_h^4 f \Vert_{L^\infty(I)}.
  \end{equation*}
\end{lemma}
\begin{proof}
  A proof can be found in \cite[Section 3.1]{SB02}
\end{proof}

\begin{lemma}\label{lemma:rhs_interpolation}
  Let $f \in H^2(I)^d$. Then we have
  \begin{equation*}
    \int_I v_h'' \cdot f'' \ dx = \int_I v_h'' \cdot 
    (\I_{h,3} f)'' \ dx
  \end{equation*}
  for all $v_h \in \S^{3,1}(\T_h)$.
\end{lemma}

\begin{proof}
  Let $f \in H^2(I)^d$ and $v_h \in \S^{3,1}(\T_h)$ arbitrary. We have
  \begin{align*}
    \int_I v_h'' \cdot f'' \ dx
    - \int_I v_h'' \cdot (\I_{h,3} f)'' \ dx
    = \int_I v_h'' \cdot (f - \I_{h,3} f)'' \ dx.
  \end{align*}
  Elementwise integration by parts and the fundamental theorem of calculus yield
  \begin{align*}
    \int_{I_i} v_h'' \cdot (f - \I_{h,3} f)'' \ dx
    &= [v_h'' \cdot (f - \I_{h,3} f)' ]_{x_{i-1}}^{x_i}
    - [v_h''' \cdot (f - \I_{h,3} f) ]_{x_{i-1}}^{x_i} \\
    &\kurz + \int_{I_i} v_h^{(4)} \cdot (f - \I_{h,3} f) \ dx.
  \end{align*}
  Now the first summand vanishes since $(\I_{h,3} f)'(x_i) = f(x_i) \ \forall i$.
  Analogously the second summand vanishes since
  $(\I_{h,3} f)(x_i) = f(x_i) \ \forall i$.
  Lastly the integral term also vanishes, since $v_h$ is elementwise $\P_3$ and therefore $D_h^4 v_h \equiv 0$. Now summation over all subintervals finishes the proof.
\end{proof}

\begin{lemma}[Inverse Estimates]\label{lemma:inverse_estimates}
  Let $I$ be an interval and $v \in \P_m,\ m \in \mathbb{N}$.
  We then have the estimates
  \begin{align*}
    \Vert v \Vert_{L^\infty(I)} 
    &\leq c\vert I \vert^{-\frac{1}{2}} \Vert v \Vert_{L^2(I)}, \\
    \vert v \vert_{H^1(I)} &\leq c\vert I \vert^{-1} \Vert v \Vert_{L^2(I)}.
  \end{align*}
  Further with the embedding $\P_2 \hookrightarrow H^1(I)'$, $v \mapsto (w \mapsto (v,w)_{L^2(I)})$ we get
  \begin{align*}
    \Vert v \Vert_{L^2(I)} &\leq c \vert I \vert^{-1} \Vert v \Vert_{H^1(I)'}.
  \end{align*}
\end{lemma}

\begin{proof}
  The first and second inequality follow directly from transformation onto the reference interval $[0,1]$ and using the equivalence of the two norms involved in finite dimensional vector spaces. For the third inequality we estimate
  \begin{align*}
    \Vert v \Vert_{L^2(I)}^2 = (v,v)_{L^2(I)} = v[v]
    \leq \Vert v \Vert_{H^1(I)'} \Vert v \Vert_{H^1(I)}.
  \end{align*}
  The asserted estimate then follows immediately from applying the second estimate.
\end{proof}

\begin{lemma}[Gagliardo--Nirenberg inequality]\label{lemma:gagliardo_nirenberg}
  Let $I = (a,b) \subset \real$ and $u \in H^2(I)$. Then we have
  \begin{equation*}
    \Vert u' \Vert_{L^2(I)}
    \leq C \Vert u \Vert_{H^2(I)}^\frac{1}{2} \Vert u \Vert_{L^2(I)}^\frac{1}{2}
    + C \Vert u \Vert_{L^2(I)}.
  \end{equation*}
\end{lemma}
\begin{proof}
  A proof can be found in \cite[Theorem 1.1]{LZ22}.
\end{proof}

\renewcommand{\thesection}{B}
\section{Inverse function theorem}
\label{appendix:inverse Function}

\begin{lemma}[Inverse function theorem]
  Let $F: X \to X'$ for a Banach space X and assume there exist $\tilde{x} \in 
  X$, $c_L, c_{inv}, \delta, \varepsilon > 0$ such that
  \begin{enumerate}
    \item $\Vert F(\tilde{x}) \Vert_{X'} \leq \delta$,
    \item $F$ is Fréchet differentiable in $B_\varepsilon(\tilde{x})$,
    \item $\Vert DF(\tilde{x})^{-1} \Vert_{L(X',X)} \leq c_{inv}$,
    \item For all $x_1, x_2 \in B_\varepsilon(\tilde x): 
    \Vert DF(x_1) - DF(x_2) \Vert_{L(X,X')}
    \leq c_L \Vert x_1 - x_2 \Vert_X$,
    \item $c_L c_{inv} \varepsilon \leq \frac{1}{2}, \ 
    2 c_{inv} \delta \leq \varepsilon $.
  \end{enumerate}
  Then there exists a unique $x \in B_\varepsilon(\tilde{x})$ with 
  $F(x) = 0$.
\end{lemma}

\begin{proof}
  We set $\widetilde y := F(\widetilde x)$. For $\varrho \in B_\varepsilon(0) \subset X$ we also define the remainder term
  \begin{align*}
    \R(\widetilde x, \varrho) := F(\widetilde x + \varrho) - F(\widetilde x) - DF(\widetilde x) \varrho
    = F(\widetilde x + \varrho) - \widetilde y - DF(\widetilde x) \varrho
  \end{align*}
  and the operator $T: B_\varepsilon(0) \to X$,
  \begin{align*}
    T(\varrho) := -DF(\widetilde x)^{-1}(\widetilde y + \R(\widetilde x,\varrho)).
  \end{align*}
  This now implies that
  \begin{align*}
    F(\widetilde x + \varrho) = 0
    \ \Leftrightarrow \  
    \R(\widetilde x, \varrho) + \widetilde y = - DF(\widetilde x) \varrho
    \ \Leftrightarrow \ 
    \varrho = -DF(\widetilde x)^{-1}(\R(\widetilde x,\varrho) + \widetilde y)
    = T(\varrho).
  \end{align*}
  So to prove the Lemma it remains to show that $T(\varrho) \in B_\varepsilon(0)$ for all $\varrho \in B_\varepsilon(0)$ and that $T$ is a contraction on $B_\varepsilon(0)$. We start with the contraction property. Let $\varrho, \varrho' \in B_\varepsilon(0)$ be chosen arbitrary. We then have that
  \begin{align*}
    \R(\widetilde x, \varrho) - \R(\widetilde x, \varrho')
    &= F(\widetilde x + \varrho) - F(\widetilde x + \varrho')
    - DF(\widetilde x)(\varrho - \varrho') \\
    &= \int_0^1 \frac{d}{dt} F(\widetilde x + (1-t) \varrho' + t \varrho) \ dt
    - DF(\widetilde x)(\varrho - \varrho') \\
    &= \int_0^1 (DF(\widetilde x + t \varrho + (1-t) \varrho') - DF(\widetilde x))(\varrho - \varrho') \ dt.
  \end{align*}
  This implies
  \begin{align*}
    \Vert \R(\widetilde x, \varrho) - \R(\widetilde x, \varrho') \Vert_{X'}
    &\leq \int_0^1 \Vert DF(\widetilde x + t\varrho + (1-t) \varrho') - DF(\widetilde x) \Vert_{L(X,X')} \Vert \varrho - \varrho' \Vert_X \ dt \\
    &\leq c_L \Vert t\varrho + (1-t) \varrho' \Vert_X \Vert \varrho - \varrho' \Vert_X \leq c_L \varepsilon \Vert \varrho - \varrho' \Vert_X.
  \end{align*}
  From this therefore we obtain
  \begin{align*}
    \Vert T(\varrho) - T(\varrho') \Vert_X
    &\leq \Vert DF(\widetilde x)^{-1} \Vert_{L(X',X)}
    \Vert \R(\widetilde x, \varrho) - \R(\widetilde x, \varrho') \Vert_{X'}
    \leq c_{inv} c_L \varepsilon \Vert \varrho - \varrho' \Vert_X.
  \end{align*}
  By assumption we have that $c_{inv}c_L \varepsilon \leq 1/2$ and thus $T$ is indeed a contraction. We further have for all $\varrho \in B_\varepsilon(0)$
  \begin{align*}
    \Vert T(\varrho) \Vert_X
    &= \Vert T(\varrho) - T(0) \Vert_X + \Vert T(0) \Vert_X
    \leq \frac{1}{2} \Vert \rho \Vert_X + \Vert DF(\widetilde x)^{-1} F(\widetilde x) \Vert_X \\
    &< \frac{\varepsilon}{2} + c_{inv} \Vert F(\widetilde x) \Vert_{X'}
    \leq \frac{\varepsilon}{2} + c_{inv} \kappa \leq \varepsilon,
  \end{align*}
  therefore we have $T(\varrho) \in B_\varepsilon(0)$. Now the Banach fixed-point theorem implies the existence of $\varrho \in B_\varepsilon(0)$ with $T(\varrho) = \varrho$ and thus $x = \widetilde x + \varrho$ is the unique solution to $F(x) = 0$ in $B_\varepsilon(\widetilde x)$.
\end{proof}

\begin{acknowledgement}
  Financial support by the German Research Foundation (DFG) via
  research unit FOR 3013 
  \textit{Vector- and tensor-valued surface PDEs} (417223351)
  is gratefully acknowledged.
\end{acknowledgement}

\printbibliography

\end{document}

%% file: macros_ds.tex
\def\del{\partial}

\def\real{\mathbb R}

\def\A{\mathcal A}
\def\G{\mathcal G}
\def\I{\mathcal I}
\def\J{\mathcal J}
\def\M{\mathcal M}
\def\N{\mathcal N}
\def\P{\mathcal P}
\def\R{\mathcal R}
\def\S{\mathcal S}
\def\T{\mathcal T}

\def\leer{\hspace{1 cm}}
\def\kurz{\hspace{0.5 cm}}

\newcommand{\Rom}[1]{\textup{\uppercase\expandafter{\romannumeral #1\relax}}}

\theoremstyle{plain}
\newtheorem{definition}{\begin{large}Definition\end{large}}[section]
\newtheorem{theorem}[definition]{\begin{large}Theorem\end{large}}
\newtheorem{remark}[definition]{\begin{large}Remark\end{large}}
\newtheorem{lemma}[definition]{\begin{large}Lemma\end{large}}

\newtheorem{proposition}[definition]{\begin{large}Proposition\end{large}}
\newtheorem{corollary}[definition]{\begin{large}Corollary\end{large}}
\newtheorem{algorithm}[definition]{\begin{large}Algorithm\end{large}}

\theoremstyle{definition}
\newtheorem{example}[definition]{\begin{large}Example\end{large}}
\newtheorem*{acknowledgement}{\begin{large}Acknowledgement\end{large}}

\theoremstyle{remark}